\documentclass[11pt]{amsart}

\usepackage{amscd,verbatim}
\usepackage{amsmath,amssymb}
\usepackage[bookmarks=false]{hyperref}
\usepackage{amsfonts,mathrsfs,amsthm,enumitem,xcolor,mathtools, cleveref}
\usepackage[hmargin=1.5in]{geometry}

\def\K{K\"ahler }
\def\KE{K\"ahler-Einstein }
\def\KR{K\"ahler-Ricci }
\def\MA{Monge-Amp\`{e}re }
\def\Ho{H\"older }

\def\a{{\alpha}}
\def\P{\mathbb{P}}

\newcommand{\de}{\partial}

\newcommand{\dbar}{\overline{\partial}}

\newcommand{\ddbar}{i \partial \overline{\partial}}
\newcommand{\Ric}{\mathrm{Ric}}
\newcommand{\Rc}{\mathrm{Rc}}
\newcommand{\Rm}{\mathrm{Rm}}

\newcommand{\tr}[2]{\textrm{tr}_{#1}{#2}}

\newcommand{\vp}{\varphi}

\renewcommand{\leq}{\leqslant}
\renewcommand{\geq}{\geqslant}

\renewcommand{\l}{\ell}
\newcommand{\owp}{\omega_{\mathrm{WP}}}
\newcommand{\ocan}{\omega_{\mathrm{can}}}
\newcommand{\kod}{\mathrm{kod}}
\newcommand{\ob}{\omega^\bullet}
\newcommand{\on}{\omega^\natural}
\newcommand{\gb}{g^\bullet}
\newcommand{\gn}{g^\natural}
\newcommand{\gcan}{g_{\mathrm{can}}}
\newcommand{\be}{\begin{equation}}
\newcommand{\ee}{\end{equation}}

\newcommand{\br}[1]{\left( {#1}\right)}
\newcommand{\lan}[1]{\langle {#1} \rangle}
\newcommand{\sqbr}[1]{\left[{#1}\right]}
\newcommand{\norm}[1]{\left\lVert{#1}\right\lVert}
\newcommand{\abs}[1]{\left|{#1}\right|}
\DeclarePairedDelimiter{\ceil}{\lceil}{\rceil}

\newcommand{\C}{\mathbb{C}}

\newcommand{\N}{\mathbb{N}}
\newcommand{\D}{\mathbb{D}}
\newcommand{\fD}{\mathfrak{D}}
\newcommand{\fG}{\mathfrak{G}}

\newcommand{\CP}{\mathbb{P}}
\newcommand{\B}{\mathcal{B}}
\newcommand{\cA}{\mathcal{A}}
\newcommand{\cG}{\mathcal{G}}
\newcommand{\cS}{\mathcal{S}}
\newcommand{\E}{\mathcal{E}}
\newcommand{\T}{\mathcal{T}}

\newcommand{\infint}{{[0, +\infty)}}

\DeclareMathOperator{\pr}{pr}
\DeclareMathOperator{\id}{Id}

\newcommand{\bb}{\mathbf{b}}
\newcommand{\ff}{\mathbf{f}}
\newcommand{\thalf}{\frac{t}{2}}
\newcommand{\dgamma}{\dot{\gamma}}
\newcommand{\deta}{\dot{\eta}}
\newcommand{\davg}[1]{{#1} - \underline{#1}}
\DeclareMathOperator{\can}{can}
\newcommand{\ddbarf}{i\de_{\ff}\dbar_{\ff}}
\newcommand{\base}{\mathrm{base}}

\numberwithin{equation}{section}

\newcounter{remark}
\newcounter{theor}
\newtheorem{theorem}{Theorem}[section]
\newtheorem{lemma}[theorem]{Lemma}

\newtheorem{conjecture}[theorem]{Conjecture}
\newtheorem{proposition}[theorem]{Proposition}
\newtheorem{question}[theorem]{Question}
\newtheorem{defn}[theorem]{Definition}
\theoremstyle{definition}
\newtheorem{rmk}[theorem]{Remark}

\title{Higher-order Ricci estimates along immortal K\"ahler-Ricci flows}

\author{Wenrui Kong}
\address{Courant Institute School of Mathematics, Computing, and Data Science, New York University, 251 Mercer St, New York, NY 10012}
\email{wenrui.kong@nyu.edu}

\begin{document}
\begin{abstract}
We study higher-order curvature estimates along \KR flows on compact K\"ahler manifolds of intermediate Kodaira dimension. We prove that away from singular fibers, the Ricci curvature is uniformly bounded in $C^1$, the Laplacian of the Ricci curvature in $C^0$, and the scalar curvature in $C^2$. We identify a geometric obstruction to higher-order curvature bounds, whose non-vanishing causes a specific third-order derivative of the Ricci curvature to blow up at rate $e^{t/2}$. Uniform $C^k$ bounds for every $k$ hold for the Ricci curvature in the isotrivial case, and for the full Riemann curvature in the torus-fibered case.

\end{abstract}

\maketitle

\section{Introduction}
	\subsection{Setup and Main Results}\label{subsec_setup}
	
	Let $(X, \omega_0)$ be a compact \K manifold. The (normalized) \KR flow on $X$ starting at $\omega_0$ is a family of \K metrics $\br{\omega^\bullet(t)}_{t \in [0, T)}$, satisfying
	\begin{equation}\label{eq_krf}
		\partial_t \omega^\bullet(t) = - \Ric(\omega^\bullet(t)) -\omega^\bullet(t), \quad \omega^\bullet(0) = \omega_0,
	\end{equation}
	for some $0 < T \leq + \infty$. In this paper we consider the case when the flow $\ob(t)$ is \textit{immortal} (i.e., we can take $T = +\infty$). By \cite{Ts,TiZ}, this happens if and only if the canonical bundle $K_X$ is nef. We assume in this paper the stronger condition that $K_X$ is semiample. The Abundance Conjecture in birational geometry predicts that the nefness of $K_X$ is equivalent to its semiampleness when $X$ is projective. The extension of this conjecture to compact \K $X$ has been proved when $\dim X \leq 3$, by \cite{CHP,DO,DO2,GPa}. 
	
	Since $K_X$ is semiample, the global sections of $K_X^{p}$ for $p \geq 1$ sufficiently divisible define a surjective holomorphic map $f: X \to B \subset \C \CP^N$, called the Iitaka fibration of $X$ (see e.g. \cite[Theorem 2.1.27]{Laz}). The normal projective variety $B$ (called the canonical model of $X$) has dimension $m$ equal to the Kodaira dimension of $X$ denoted by $\kod(X)$. Let $S \subset X$ be the preimage of the union of the singular locus of $B$ and the set of singular values of $f$. Then $f: X \setminus S \to B \setminus f(S)$ is a proper holomorphic submersion with $n$-dimensional connected Calabi-Yau fibers $\{X_z = f^{-1}(z) \mid z \in B \setminus f(S)\}$, where $n = \dim X - m$.

	Since the behavior of the flow in the extremal cases $\kod(X) = 0$ or $\kod(X) = \dim X$  has been completely understood by \cite{Cao,Ts,PS,TiZ,TZ}, in this paper we assume $X$ has \textit{intermediate Kodaira dimension}: $0 < \kod(X) < \dim X$, so $m, n > 0$. In this case, the limiting behavior of the flow has been understood: The foundational works \cite{ST,ST2} of Song-Tian established the existence of a closed positive $(1,1)$-current $\ocan$ on $B$ (called the canonical metric), which restricts to a smooth metric on $B \setminus f(S)$ solving the twisted \KE equation
	\begin{equation}
		\Ric(\ocan) = - \ocan + \owp,
	\end{equation}
	such that $\ob(t)$ converges to $f^*\ocan \in -2\pi c_1(X)$ as currents on $X$.	The semipositive Weil-Petersson form $\owp$ describes the variation of complex structures on $X_z$ with $z$. To simplify notation, we will use $\ocan$ to denote also $f^*\ocan$ on $X$. In \cite{HLT}, Hein-Lee-Tosatti proved the conjecture by Song-Tian (\cite{ST,ST2}) that the convergence $\ob(t) \to \ocan$ happens also in the locally smooth topology on $X \setminus S$. Recently, Lee-Tosatti-Zhang proved in \cite{LTZ} that $(X, \ob(t))$ converges to the metric completion of $(B \setminus f(S), \ocan)$ in the Gromov-Hausdorff topology.
	
	In this paper we investigate higher-order curvature estimates along the flow. In \cite{ST3}, Song-Tian proved that the scalar curvature of $\ob(t)$ is uniformly bounded on $X$. They also conjectured (see \cite[Conjecture 4.7]{Ti2}) the existence of a uniform bound on the Ricci curvature of $\ob(t)$ on any $K \Subset X \setminus S$. Hein-Lee-Tosatti confirmed this conjecture in \cite{HLT} as a consequence of the local asymptotic expansion of $\ob(t)$ developed in \cite{HT2,HT3,HLT}.	
	A natural follow-up question is whether we can extend such uniform bounds to the covariant derivatives of $\Ric(\ob(t))$ of any order. The answer may provide insight into the rate of convergence of the flow $\ob(t)$ and its parabolically rescaled flows (see \cite[Theorem 1.3]{TZ}).

	We first describe our setup. Following \cite{ST2}, we construct a closed real $(1,1)$-form $\omega_F=\omega_0+\ddbar\rho$ on $X\backslash S$ with $\rho \in C^\infty(X \setminus S)$, such that for every $z\in B\backslash f(S)$, $\omega_F|_{X_z}$ is the unique Ricci-flat \K metric on $X_z$ cohomologous to $\omega_0|_{X_z}$. Then the closed real $(1, 1)$-form
	\begin{equation}
		\omega^\natural(t):=(1-e^{-t})\omega_{\mathrm{can}}+e^{-t}\omega_F
	\end{equation}
	eventually becomes a \K metric on any $K \Subset X \setminus S$, and we can write $\omega^\bullet(t) = \omega^\natural(t) + \ddbar \varphi(t)$  on $X \setminus S$ such that the potentials $\{\varphi(t) \mid t \geq 0\}$ solve the parabolic complex \MA equation
	\begin{equation}\label{eq_krf_MA}
		\left\{
		\begin{aligned}
			&\frac{\de}{\de t}\vp(t)=\log\frac{e^{nt}(\omega^\natural(t)+\ddbar\vp(t))^{m+n}}{\binom{m+n}{n}\omega_{\rm can}^m\wedge\omega_F^n}-\vp(t),\\
			&\vp(0)=-\rho,\\
			&\omega^\natural(t)+\ddbar\vp(t)>0,
		\end{aligned}
		\right.
	\end{equation}
	as an equivalent formulation of the \KR flow \eqref{eq_krf} (see e.g. \cite[\S5.7]{To} and \cite[\S3.1]{TWY}).
	
	We are interested in estimates over $K \Subset X \setminus S$, and $f$ is differentiably a locally trivial fiber bundle over $B \setminus f(S)$ with compact fibers (by Ehresmann's lemma). Thus we can work locally and assume that our base $B$ is now simply the Euclidean unit ball in $\C^m$, and $f: B \times Y \to B$ is just the projection $\pr_B$ onto the first factor, where $Y$ is a closed manifold and $B \times Y$ is equipped with the complex structure $J$ induced from $X$ (not necessarily a product) such that $f$ is $(J, J_{\C^m})$-holomorphic. Each fiber $X_z$ now is written as $(\{z\} \times Y, J|_{\{z\} \times Y} =: J_z)$, and each Ricci-flat \K metric $\omega_F|_{X_z}$ written as a Riemannian metric $g_{Y, z}$ on $\{z\} \times Y$, which we extend trivially to $B \times Y$ and use these to define a family of shrinking Riemannian product metrics
	\begin{equation}\label{eq_g_z(t)}
		g_z(t) = g_{\C^m} + e^{-t} g_{Y, z},
	\end{equation}
	on $B \times Y$. We then define (as in \cite{HLT}) a $t$-independent connection $\D$ on $B \times Y$ by
	\begin{equation}
		\br{\D \eta}(x) := \br{\nabla^{\pr_B(x)} \eta}(x),
	\end{equation}
	where $\nabla^z$ denotes the Levi-Civita connection of $g_z(t)$ for each $z \in B$.
	The metrics $\gb(t)$, $\gn(t)$, $g_z(t)$ are uniformly equivalent on $B \times Y$ by \cite{FZ}, and we use such locally defined $\D$-derivative to mimic the covariant derivative of $\ob(t)$. For simplicity, we write $g(t):= g_0(t)$ when $z = 0 \in B$.

	Delicate applications of the above-mentioned asymptotic expansion of $\ob(t)$, which locally decomposes the potential $\varphi(t)$ into a sequence of components with increasing rate of decay, enable us to prove:

	\begin{theorem}\label{thm_nabla^1_ric}
		Let $(X,\omega_0)$ be a compact K\"ahler manifold with $K_X$ semiample and intermediate Kodaira dimension $0 < m < \dim X$, and let $\omega^\bullet(t)$ solve \eqref{eq_krf}. Given any $K \Subset X \setminus S$, there exists a constant $C_{K}$ such that 
		\begin{equation}
			\sup_{K \times [0, +\infty)} \abs{\nabla^{\omega^\bullet(t)} \Ric(\omega^\bullet (t))}_{\omega^\bullet(t)} \leq C_{K}.
		\end{equation}
	\end{theorem}

	\begin{theorem}\label{thm_nabla^2_scalar}
		As in \Cref{thm_nabla^1_ric}, we have
		\begin{equation}
			\sup_{K \times [0, +\infty)} \abs{\nabla^{\omega^\bullet(t), 2} R(\omega^\bullet (t))}_{\omega^\bullet(t)} \leq C_{K},
		\end{equation}
		where $R$ denotes the scalar curvature.
	\end{theorem}

	\begin{theorem}\label{thm_Delta_ric}
		As in \Cref{thm_nabla^1_ric}, we have
		\begin{equation}
			\sup_{K \times [0, +\infty)} \abs{\Delta^{\omega^\bullet(t)} \Ric(\omega^\bullet (t))}_{\omega^\bullet(t)} \leq C_{K},
		\end{equation}
		where $\Delta$ denotes the rough Laplacian.
	\end{theorem}
	
	These follow from the local estimates on the $\D$-derivatives of the Ricci curvature:
	\begin{theorem}\label{thm_D^1_ric}
		Locally on $B \times Y$, there exists a constant $C$ such that
		\begin{equation}
			\sup_{B \times Y \times [0, +\infty)} \abs{\D \Ric(\ob(t))}_{\ob(t)} \leq C.
		\end{equation}
	\end{theorem}
	
	\begin{theorem}\label{thm_D^2_ric}
		As in \Cref{thm_D^1_ric}, we have
		\begin{equation}
			\sup_{B \times Y \times [0, +\infty)} \br{\abs{\D^2 \Ric(\ob(t))}_{\ob(t)} + \abs{\partial_t \Ric(\ob(t))}_{\ob(t)}} \leq C.
		\end{equation}
	\end{theorem}
	
	The transition from estimates on $\D^2 \Ric$ to $\nabla^{\bullet, 2} \Ric$, however, is hindered by the unknown variation of the semiflat form $\omega_F$ and the complex structure $J$ (viewed as a tensor) on $B \times Y$, which depend on the arbitrary initial data $(X, \omega_0)$. The complex structure $J$ is in general not a product on $B \times Y$. Nevertheless, if all the regular fibers $X_z$ are pairwise biholomorphic (we call such $f$ an \textit{isotrivial} fibration), so that $J$ can be made a product locally on $B \times Y$ by the Fischer-Grauert theorem, then by \cite{FL} we in fact have uniform bounds on the covariant derivatives of the Ricci curvature of every order away from singular fibers. By \cite{Cao,Ts,PS,TiZ,TZ}, the same estimates hold when $\kod(X) = 0$ or $\kod(X) = \dim X$, both of which can be viewed as extremal cases of an isotrivial fibration. 
	On the other hand, if the generic fibers are tori (the complex structures $J|_{X_z}$ may vary with $z$ in this case), we have in fact uniform control on the covariant derivatives of the full Riemann curvature tensor of every order on $K \Subset X \setminus S$.

	The third order $\D$-derivative of the Ricci curvature witnesses a probable blow-up:
	
	\begin{theorem}\label{thm_D^3_Ric}
		Under the assumptions above, if the function $\cS$ defined by
		\begin{equation}\label{eq_S}
			\frac{m}{n+1} \ocan^{m-1} \wedge \omega_F^{n+1} = \cS \ocan^m \wedge \omega_F^n
		\end{equation}
		is not fiberwise constant in $B \times Y$, then
		\begin{equation}\label{eq_D2_Ric}
			e^{-\thalf} \br{\sup_{B \times Y } \abs{\D^3 \Ric(\ob(t))}_{\ob(t)}} = \Theta(1).
		\end{equation}
	\end{theorem}
	
	The ``geodesic curvature'' $\cS$ defined in \eqref{eq_S} is constant along a fiber $(X_z, \omega_F|_{X_z})$ if and only if the harmonic representatives of the Kodaira-Spencer classes for the Iitaka fibration $f$ have constant total $\omega_F|_{X_z}$-length measured in $\ocan(z)$ (see \cite[\S 5]{HT3} and the references therein), which may well not be the case due to the asymptotically cylindrical gluing construction of K3 surfaces (\cite{He1,CC}). Therefore, we make the following conjecture:

	\begin{conjecture}\label{conj_main}
		Under the assumptions above, given $K \Subset X \setminus S$, the quantity
		\begin{equation}
			\abs{\nabla^{\omega^\bullet(t), 3} \Ric(\omega^\bullet (t))}_{\omega^\bullet(t)} 
		\end{equation}
		is in general not bounded over ${K \times [0, +\infty)}$.
	\end{conjecture}

	\begin{question}
		Under the assumptions above, given $K \Subset X \setminus S$, in which $C^{k, \alpha}$ norms between $C^1$ and $C^3$ over $(K, \ob(t))$ is the Ricci curvature of $\ob(t)$ uniformly bounded in $t \geq 0$? Does the answer depend only on the variation of complex structures $J|_{X_z}$ with $z \in B \setminus f(S)$?
	\end{question}

	The Ricci curvature bounds near the singular fibers of $f$ remain an open problem. We refer interested readers to \cite[Remark 1.4]{HLT}.

	\subsection{Paper Outline}
	In \Cref{section_prel}, we briefly present the parabolic \Ho norm and the asymptotic expansion theorem from \cite{HLT}. In \Cref{section_order_1}, we prove \Cref{thm_D^1_ric} and use it to derive \Cref{thm_nabla^1_ric}. In \Cref{section_order_2}, we prove \Cref{thm_D^2_ric}, followed by \Cref{thm_nabla^2_scalar,thm_Delta_ric}, and discuss the difficulties in obtaining second-order Ricci bound. \Cref{section_order_3} is devoted to \Cref{thm_D^3_Ric}, including how its assumption on $\cS$ can be realized. Finally, we discuss special cases in \Cref{section_special} including when the Iitaka fibration is isotrivial or torus-fibered, and when the Kodaira dimension takes the extremal values $0$ or $\dim X$.

	\subsection{Acknowledgements} 
		The author would like to thank Professor Valentino Tosatti, his Ph.D. advisor, for the motivation and helpful discussions on this paper. The author is also grateful to H.-J. Hein and M.-C. Lee for their valuable suggestions. The author thanks P. Engel and M. Mauri for discussions about Remark \ref{rmk_moduli}, as well as F.-T.-H. Fong, Y. Li, J. Zhang, and Y. Zhang for their comments.

\section{Preliminaries}\label{section_prel}
	We lay out the machinery of asymptotic expansion for the immortal \KR flow developed in \cite{HLT}, which is a parabolic adaptation of \cite{HT2,HT3}, in alignment with our setup in \Cref{subsec_setup}.

	\subsection{Known Estimates}
	We first recall some known estimates for the \KR flow \eqref{eq_krf} and its equivalence \eqref{eq_krf_MA}, for future reference. Throughout this paper, $O(\cdot)$, $o(\cdot)$, and $\Theta(\cdot)$ will always denote the asymptotic behavior as $t \to +\infty$, with estimates always uniform in space $B \times Y$ up to shrinking the Euclidean ball $B \subset \C^m$.
	\begin{lemma}\label{lem_known_est}
		On $B \times Y \times [0, +\infty)$, we have
		\begin{equation}\label{eq_unif_equiv_bullet_natural}
			C^{-1} \omega^\natural(t) \leq \omega^\bullet(t) \leq C \omega^\natural(t),
		\end{equation}
		\begin{equation}\label{eq_phi_o(1)}
			\abs{\varphi(t)} = o(1),
		\end{equation}
		\begin{equation}\label{eq_dot_phi_o(1)}
			\abs{\dot{\varphi}(t)} = o(1),
		\end{equation}
		\begin{equation}\label{eq_R_O(1)}
			\abs{R(\omega^\bullet(t))} = O(1),
		\end{equation}
		\begin{equation}\label{eq_ddot_phi_O(1)}
			\abs{\dot{\varphi}(t) + \ddot{\varphi}(t)} = O(1).
		\end{equation}
	\end{lemma}
	\begin{proof}
		\eqref{eq_unif_equiv_bullet_natural} is proved in \cite{FZ} (see also \cite{ST,To4}). \eqref{eq_phi_o(1)} and \eqref{eq_dot_phi_o(1)} are proved in \cite[Lemma 3.1]{TWY}. 
		\eqref{eq_R_O(1)} is the main theorem of \cite{ST3}, and this implies \eqref{eq_ddot_phi_O(1)} thanks to the relation \cite[p.345]{To}
		\begin{equation}\label{eq_scal_evol}
			\dot{\vp}(t)+\ddot{\vp}(t)=-R(\ob(t))-m.
		\end{equation}
	\end{proof}
	
	In our setting, $\br{g_{Y,z}}_{z\in B}$ is a smooth family of Riemannian metrics on $Y$, so (up to shrinking $B$ slightly) we can find $\Lambda>1$ such that for all $z \in B$, we have
	\begin{equation}\label{equiv-product-ref}
		\left\{
		\begin{array}{ll}
			\Lambda^{-1} g_{Y,0}\leq g_{Y,z}\leq \Lambda g_{Y,0},\\
			\Lambda^{-\frac{1}{2}} \leq \mathrm{inj}(Y,g_{Y,z}) \leq \mathrm{diam}(Y,g_{Y,z})\leq \Lambda^{\frac{1}{2}}.
		\end{array}
		\right.
	\end{equation}
	Therefore, $\Lambda^{-1} g(t) \leq g_z(t) \leq \Lambda g(t)$. Combined with \eqref{eq_unif_equiv_bullet_natural}, we derive:
	\begin{lemma}\label{lem_metric_equiv}
		On $B \times Y \times \infint$, the metrics $g^\bullet(t)$, $g^\natural(t)$, and $\{g_z(t) \mid z \in B\}$, are uniformly equivalent as $t \to +\infty$. In particular, the norms measured with respect to these metrics are uniformly comparable as $t \to +\infty$.
	\end{lemma}
	We will use \Cref{lem_metric_equiv} many times in this paper without explicit reference.

	\subsection{Parabolic \Ho Norm}\label{sec_para_holder}
	We define the local spatial $\D$-derivative and its parabolic version, denoted by $\fD$, and use these to define the parabolic \Ho norm.
	
	For each $z \in B \subset \C^m$, let $\nabla^z$ be the Levi-Civita connection of the product metric $g_z(t)$ (see \eqref{eq_g_z(t)}) on $B \times Y$, which is $t$-independent.
	\begin{defn}\label{def_D}
		Define a connection $\D$ on $B \times Y$ by
		\begin{equation}
			\br{\D \eta}(x) = \br{\nabla^{\pr_B(x)} \eta}(x),
		\end{equation}
		for all tensor fields $\eta$ on $B \times Y$ and $x \in B \times Y$.
	\end{defn}

	For a detailed discussion of the properties of $\D$, we refer readers to \cite[\S 2.1]{HT3}.  
	
	\begin{defn}
		Given a curve $\gamma$ in $B\times Y$ joining $a$ to $b$, let $\P^\gamma_{ab}$ denote the $\D$-parallel transport from $a$ to $b$ along $\gamma$. $\gamma$ is called a $\mathbb{P}$-geodesic if $\dot{\gamma}$ is $\D$-parallel along $\gamma$.  
	\end{defn}

	Two examples of $\P$-geodesics are \textit{horizontal} paths $(z(t),y_0)$ where $z(t)$ is an affine segment in $\C^m$, and \textit{vertical} paths $(z_0,y(t))$ where $y(t)$ is a geodesic in $(\{z_0\}\times Y, g_{Y,z_0})$. These are the only $\P$-geodesics that we will use in the paper, as every pair of points in $B\times Y$ can be joined by concatenating two of these $\P$-geodesics with the vertical one minimal. We may also write $\P_{ab}$ instead of $\P^\gamma_{ab}$ if the $\P$-geodesic $\gamma$ joining $a$ and $b$ is not emphasized.
	
	\begin{defn}
		Given a time-dependent tensor field $\eta$ and $k \in \N$, define
	\begin{equation}
		\fD^k\eta:=\sum_{p+2q=k}\D^p \partial_t^q \eta,
	\end{equation}
	as a formal sum of tensors of different types. Given any Riemannian product metric $g$ on $B\times Y$, define
	\begin{equation}
		|\fD^k\eta|^2_g := \sum_{p+2q=k}|\D^p \partial_t^q \eta|^2_g.
	\end{equation}
	\end{defn}

	\bigskip
	We now define the parabolic \Ho norm on $B\times Y\times [0,+\infty)$ associated to the connection $\D$.
	Given $p=(z,y)\in B\times Y$, $t\geq 0$, $0<R\leq \sqrt{t}$, and (shrinking) product metrics $g_\zeta(\tau)=g_{\mathbb{C}^m}+e^{-\tau} g_{Y,\zeta}$, we define the parabolic domain
	\begin{equation}
		Q_{g_\zeta(\tau),R}(p,t) := B_{\mathbb{C}^m}(z,R)\times  B_{e^{-\tau}g_{Y,\zeta}}(y,R)\times [t-R^2,t].
	\end{equation}
	The parabolic domain with respect to any other product metric is defined analogously. 
	
	\begin{defn}
		For any $0< \alpha <1$, $R>0$, $p\in B\times Y$,  $t\geq 0$ and smooth tensor field $\eta$ on $B\times Y\times [t-R^2,t]$, given a product metric $g$ (such as $g=g_{z}(\tau)$ for some $z \in B$ and $\tau\geq 0$),  define the parabolic \Ho seminorm by
		\begin{equation}\label{kzk}
			\sqbr{\eta}_{\a,\a/2, Q_{g,R}(p,t),g}:=\sup \left\{ \frac{|\eta(x,s)-\P_{x'x}\eta(x',s')|_g}{(d^g(x,x')+|s-s'|^{\frac{1}{2}})^\a}\right\},
		\end{equation}
		where the supremum is taken among all $(x,s)$ and $(x',s')$ in $Q_{g,R}(p,t)$ in which $x$ and $x'$ are either horizontally or vertically joined by a $\P$-geodesic. In addition, for any $k \in \N$, define the parabolic \Ho norm by
		\begin{equation}
			\norm{\eta}_{k, \alpha, Q_{g, R}(p, t), g} := \sum_{j=0}^k \norm{\fD^j \eta}_{\infty, Q_{g,R}(p,t),g} + \sqbr{\fD^k \eta}_{\a,\a/2, Q_{g,R}(p,t),g}.
		\end{equation}
	\end{defn}
	
	In this paper, we always use $g=g_z(\tau)$ for the parabolic \Ho (semi)norms, and often we take $z = 0 \in B$. Observe that for each fixed $R>0$, there exists $\tau_0 \geq 0$ such that
	\begin{equation}
		B_{\mathbb{C}^m}(z,R)\times  B_{e^{-\tau}g_{Y,z}}(y,R)=B_{\mathbb{C}^m}(z,R)\times  Y,
	\end{equation}
	for all $\tau>\tau_0$. Therefore, since we are only interested in asymptotic behaviors as $\tau \to +\infty$, we can simplify the parabolic domains to
	\begin{equation}
		Q_R(z,t):=B_{\mathbb{C}^m}(z,R)\times Y \times [t-R^2,t].
	\end{equation}

	\subsection{Asymptotic Expansion}
	We now state the asymptotic expansion for $\omega^\bullet(t)$ given in \cite[\S 4]{HLT}.
	\begin{defn}
		Given any function $u$ on $B \times Y$, let $\underline{u}$ denote its fiberwise average:
	\begin{equation}
		\underline{u}(z) := \frac{\int_{\{z\} \times Y} u(z, \cdot) dV_{g_{Y, z}}}{\int_{\{z\} \times Y}  dV_{g_{Y, z}}}, \quad z \in B.
	\end{equation}
	\end{defn}

	\begin{theorem}[Asymptotic Expansion, \cite{HLT}]\label{thm_asymp_expa}
		For all $j, k \in \N$, $0 \leq 2j \leq k$, $z \in B$, there exists $B' = B_{\C^m}(z, R) \Subset B$ such that on $B' \times Y$ we have a decomposition
		\begin{equation}\label{eq_asymp_expa}
			\omega^\bullet(t)=\omega^\natural(t)+\gamma_0(t)+\gamma_{1,k}(t)+\dots+\gamma_{j,k}(t)+\eta_{j,k}(t),
		\end{equation}
		with the terms in \eqref{eq_asymp_expa} given by
		\begin{equation}\label{eq_gamma_formula}
			\gamma_0(t)=\ddbar \underline{\varphi},\quad \gamma_{i,k}(t)=\ddbar \sum_{p=1}^{N_{i,k}}\mathfrak{G}_{t,k}\left(A_{i,p,k}(t),G_{i,p,k} \right),
		\end{equation}
		for $1 \leq i \leq j$, and $\eta_{j, k}(t)$ is hence the remainder, where
		\begin{enumerate}
			\item $\cG_{j, k} := \{G_{i, p, k} \mid 1 \leq i \leq j, 1 \leq p \leq N_{i, k}\}$ is a family of smooth functions on $B' \times Y$ which are fiberwise $L^2$ orthonormal and have fiberwise average zero;
			
			\item $\cA_{j, k} := \{A_{i, p, k}(t) \mid 1 \leq i \leq j, 1 \leq p \leq N_{i, k}\}$ is a family of smooth functions on $B' \times [0, +\infty)$, identified with the trivial extensions to $B' \times Y \times [0, +\infty)$;
			
			\item  $\{\mathfrak{G}_{t,k}\left(A_{i,p,k}(t),G_{i,p,k} \right) \mid 1 \leq i \leq j, 1 \leq p \leq N_{i, k}\}$ is a family of smooth functions on $B' \times Y \times \infint$ which have fiberwise average zero.
		\end{enumerate}
		Moreover, the following estimates hold. For all $\alpha \in (0, 1)$ and $r < R$, there is $C > 0$ such that for all $t \geq 0$, 
		\begin{equation}\label{eq_expa_gamma_0_infty}
			\|\mathfrak{D}^\iota\gamma_{0}\|_{\infty, Q_r(z,t),g_{\mathbb{C}^m}} = o(1), \quad \forall 0\leq \iota\leq 2j,
		\end{equation}
		\begin{equation}\label{eq_expa_gamma_0_holder}
			\quad [\mathfrak{D}^{2j}\gamma_{0}]_{\a,\a/2, Q_r(z,t),g_{\mathbb{C}^m}} \leq C,
		\end{equation}
		\begin{equation}\label{eq_expa_A_i_infty}
			\begin{aligned}
				\|\mathfrak{D}^\iota A_{i,p,k}\|_{\infty,Q_r(z,t),g_{\C^m}}
				& \leq Ce^{-(2i+\a)(1-\frac{\iota}{2j+2+\a})\frac{t}2}, \\ 
				& \quad \forall 0\leq \iota\leq 2j+2,1\leq i \leq j,1\leq p\leq N_{i,k},
			\end{aligned}
		\end{equation}
		\begin{equation}\label{eq_expa_A_i_infty_2}
			\begin{aligned}
				\|\mathfrak{D}^{2j+2+\iota} A_{i,p,k}\|_{\infty,Q_r(z,t),g_{\C^m}}& \leq Ce^{\left(-\frac{\a(2i+\a)}{\iota+\a}(1-\frac{2j+2}{2j+2+\a})+\frac{\iota^2}{\iota+\a}\right) \frac{t}{2}}, \\
				& \quad \forall 0\leq \iota\leq 2k,\;\, 1\leq i \leq j,\;\,1\leq p\leq N_{i,k},
			\end{aligned}
		\end{equation}
		\begin{equation}\label{eq_expa_A_i_holder}
			\sqbr{\fD^{2j+2+\iota} A_{i, p, k}}_{\alpha, \alpha/2, Q_r(z, t), g_{\C^m}} \leq C e^{\iota \frac{t}{2}}, \quad \forall -2 \leq \iota \leq 2k, 1 \leq i \leq j, 1 \leq p \leq N_{i, k}, 
		\end{equation}
		\begin{equation}\label{eq_expa_eta_infty}
			\|\mathfrak{D}^\iota \eta_{j,k}\|_{\infty,Q_r(z,t),g(t)}
			\leq Ce^{\frac{\iota-2j-\a}{2}t}, \quad \forall 0\leq \iota\leq 2j,\\
		\end{equation}
		\begin{equation}\label{eq_expa_eta_holder}
			[\mathfrak{D}^{2j}\eta_{j,k}]_{\a,\a/2, Q_r(z,t),g(t)}\leq C,
		\end{equation}
		where $Q_r(z,t)=\left(B_{\mathbb{C}^m}(z,r)\times Y\right)\times [t-r^2,t]$.
	\end{theorem}
	
	\begin{rmk}
		We do not list in Theorem \ref{thm_asymp_expa} all the asymptotic estimates established in \cite[Theorem 4.2]{HLT}, but only what we will need. $\fG_{t, k}$ is the $t$-dependent approximate Green operator defined in \cite[\S3.2]{HT3}.
	\end{rmk}
	
	\begin{rmk}
		The proof of Theorem \ref{thm_asymp_expa} in \cite{HLT} is for each fixed $k \in \N$ by induction on $j$. The induction step from $j-1$ to $j$ achieves the following up to shrinking $R$:
		\begin{enumerate}
			\item extend $\cG_{j-1, k}$ to $\cG_{j, k}$ by adding functions $G_{j, p, k}$ for $1 \leq p \leq N_{j, k}$;
			
			\item extend $\cA_{j-1, k}$ to $\cA_{j, k}$ by adding functions $A_{j, p, k}$ for $1 \leq p \leq N_{j, k}$;
			
			\item now that $\gamma_{j-1, k}(t)$ is defined by \eqref{eq_gamma_formula}, split $\eta_{j-1, k} = \gamma_{j-1, k} + \eta_{j, k}$.
		\end{enumerate}
		Therefore, when we (as we shall below) apply Theorem \ref{thm_asymp_expa} for some fixed sufficiently large even $k \in \N$ but different $j \leq k/2$, the family of functions $\{G_{i, p, k}\}$ and $\{A_{i, p, k}\}$ are without ambiguity always those contained within $\cG_{k/2, k}$ and $\cA_{k/2, k}$, respectively, as long as we fix a sufficiently small $R > 0$ that works for all $j \leq k/2$ (hence determined only by $k$ and $z \in B$) and work on $B_{\C^m}(z, R) \times Y \times [0, +\infty)$. 
	\end{rmk}
	
	We also have the following quasi-explicit formulae for $\gamma_{i, k}$.
	\begin{lemma}\label{lem_green}
		In Theorem \ref{thm_asymp_expa}, for all $1 \leq i \leq j$, $1 \leq p \leq N_{i, k}$, we can write
		\begin{equation}\label{eq_Green_t_k}
			\mathfrak{G}_{t,k}\left(A_{i,p,k}(t),G_{i,p,k} \right)
			= \sum_{\iota = 0}^{2k} \sum_{r = \ceil{\frac{\iota}{2}}}^k
			e^{-rt} \Phi_{\iota, r}(G_{i, p, k}) \circledast \D^\iota A_{i, p, k}(t),
		\end{equation}
		where
		\begin{enumerate}
			\item $\Phi_{\iota, r}(G_{i,p,k})$ are $t$-independent smooth functions on $B' \times Y$, for all $0 \leq \iota \leq 2k$, $\ceil{\frac{\iota}{2}} \leq r \leq k$, and
			\begin{equation}\label{eq_Phi_00}
				\Phi_{0, 0} (G) = \br{\Delta^{\omega_F|_{\{\cdot\} \times Y}}}^{-1} G,
			\end{equation}
			for all smooth functions $G$ on $B' \times Y$ having fiberwise average zero;
			
			\item $\circledast$ denotes some tensorial contraction, possibly involving $t$-independent tensor fields pulled back from the base (which we omit for convenience throughout this paper: see Remark \ref{rmk_circledast} below).
		\end{enumerate}
		Thus for all $1 \leq i \leq j$,
		\begin{equation}\label{eq_gamma_i_k}
			\gamma_{i,k}(t)=\ddbar\sum_{p=1}^{N_{i,k}}\sum_{\iota=0}^{2k}\sum_{r=\lceil\frac{\iota}{2}\rceil}^{k}e^{-r t}\Phi_{\iota,r}(G_{i,p,k})\circledast\D^\iota A_{i,p,k}(t).
		\end{equation}
		Moreover, for all $0 \leq q \leq k$, 
		\begin{equation}\label{eq_D^q_gamma_i_k}
			\D^q\gamma_{i,k}=\sum_{p=1}^{N_{i,k}}\sum_{\iota=0}^{2k}\sum_{r=\lceil\frac{\iota}{2}\rceil}^{k}\sum_{s=0}^{q+1}\sum_{i_1+i_2=s+1}e^{-r t}(\D^{q+1-s}J)\circledast\D^{i_1}\Phi_{\iota,r}(G_{i,p,k})\circledast\D^{i_2+\iota} A_{i,p,k},
		\end{equation}
		\begin{equation}\label{eq_fD^q_gamma_i_k}
			\fD^q\gamma_{i,k}=\sum_{p=1}^{N_{i,k}}\sum_{\iota=0}^{2k}\sum_{r=\lceil\frac{\iota}{2}\rceil}^{k}\sum_{s=0}^{q+1}\sum_{i_1+i_2=s+1}e^{-r t}(\D^{q+1-s}J)\circledast\D^{i_1}\Phi_{\iota,r}(G_{i,p,k})\circledast\fD^{i_2+\iota} A_{i,p,k}.
		\end{equation}
		
	\end{lemma}
	
	\begin{proof}
		\eqref{eq_Green_t_k} along with its clarifications follows from \cite[Lemma 3.8]{HT3}. \eqref{eq_gamma_i_k} is a simple combination of \eqref{eq_gamma_formula} and \eqref{eq_Green_t_k}. Then \eqref{eq_D^q_gamma_i_k} and \eqref{eq_fD^q_gamma_i_k} follow from \eqref{eq_gamma_i_k} using $\ddbar = \frac{1}{2} dd^c$ (see \cite[(5.10)]{HT3} and \cite[(5.3)]{HLT}).
	\end{proof}
	
	\begin{rmk}\label{rmk_circledast}
		Since $t$-independent tensor fields pulled back from the base have constant $g_z(t)$-norms, and the space of such tensors are closed under $\D$ and $\fD$, we are justified to hide them in $\circledast$ when we derive asymptotic bounds on $\D^q \gamma_{i,k}$.
	\end{rmk}

\section{First-Order Ricci Estimates}\label{section_order_1}
As explained in the Introduction, in this paper we work locally away from the singular fibers, and study the \KR flow evolving on $B \times Y \times \infint$ for some Euclidean ball $B \subset \C^m$. Together with the discussions in \Cref{section_prel}, we are allowed to assume for simplicity that the parabolic domain is always equal to $B \times Y \times [t-1, t]$ (or $B \times [t-1, t]$ for objects that live on the base), which we will omit in the notation for parabolic \Ho (semi)norms. The ball $B$ and the interval $[t-1, t]$ will also shrink slightly every time we use parabolic interpolation. To simplify notation even further, if not explicitly declared, all norms and seminorms will be measured with respect to the shrinking product metric $g(t)$ (which is equivalent to those measured with respect to $\gb(t)$ by \Cref{lem_metric_equiv}), and we will write parabolic seminorms as $\sqbr{\cdot}_{\alpha, \dots}$ instead of $\sqbr{\cdot}_{\alpha, \alpha/2, \dots}$.

\cite[Theorem 1.3]{HLT} and Theorem \ref{thm_nabla^1_ric} translate in our local setting  to the following, respectively:
	\begin{equation}\label{eq_bound_ric}
		\sup_{B \times Y \times \infint} \abs{\Ric(\ob(t))}_{g^\bullet(t)} \leq C,
	\end{equation}
	and
	\begin{equation}\label{eq_bound_nabla_ric}
		\sup_{B \times Y \times \infint} \abs{\nabla^{\ob(t)} \Ric(\ob(t))}_{g^\bullet(t)} \leq C.
	\end{equation}
We can easily see that \eqref{eq_bound_nabla_ric}, and hence \Cref{thm_nabla^1_ric}, follows from \Cref{thm_D^1_ric} and the following estimate:
\begin{theorem}\label{prop_D_nabla_b_Ric}
	Under all the assumptions above, there exists $C > 0$ such that
	\begin{equation}
		\sup_{B \times Y \times \infint} \abs{\br{\nabla^{\ob(t)} - \D} \Ric(\ob(t))}_{g^\bullet(t)} \leq C.
	\end{equation}
\end{theorem}

To proceed, we apply Theorem \ref{thm_asymp_expa} with any fixed even integer $k \geq 4$ and $j = 1, 2$. Up to shrinking $B$, we can write on $B \times Y \times \infint$
\begin{equation}\label{eq_expa_j=1,2}
	\ob = \on + \gamma_0 + \gamma_{1, k} + \eta_{1, k} = \on + \gamma_0 + \gamma_{1, k} + \gamma_{2, k} + \eta_{2, k}.
\end{equation}

\subsection{Proof of \texorpdfstring{\Cref{thm_D^1_ric}}{Theorem 1.4}}

	We make the following preparations. First we improve the estimates in \eqref{eq_expa_A_i_infty} and \eqref{eq_expa_A_i_infty_2}.
	\begin{lemma}\label{lem_interp_fD_A_1}
		For all $1 \leq p \leq N_{1, k}$, 
		\begin{equation}\label{eq_interp_fD_A_1}
			\abs{\fD^\l A_{1, p, k}} \leq 
			\begin{cases}
				C e^{-\frac{t}{2} \br{4-\frac{2 \l}{4+\alpha}}}, & \forall 0 \leq \l \leq 4, \\
				C e^{-\frac{t}{2} \br{4 - \frac{\l(\l-2)}{\l+\alpha}}}, & \forall 4 \leq \l \leq 2k+ 6.
			\end{cases}
		\end{equation}
	\end{lemma}
	\begin{proof}
		The case $\l = 0$, i.e. $\abs{A_{1, p, k}} \leq C e^{-2t}$, was proved in \cite[Proposition 5.3]{HLT}. By \eqref{eq_expa_A_i_holder}, $[\fD^4 A_{1, p, k}]_\alpha \leq C e^{-t}$. We can interpolate between these two estimates to get (up to shrinking $B$) the first case of \eqref{eq_interp_fD_A_1}. 
		Similarly, by \eqref{eq_expa_A_i_holder}, $[\fD^\l A_{1, p, k}]_\alpha \leq C e^{\frac{t}{2} \br{\l-6}}$ for all $4 \leq \l \leq 2k+6$, between which and $\abs{A_{1, p, k}} \leq C e^{-2t}$ we interpolate to get the second case of \eqref{eq_interp_fD_A_1}. 
	\end{proof}

	\begin{lemma}\label{lem_interp_fD_A_2}
    For all $1 \leq p \leq N_{2, k}$, 
    \begin{equation}\label{eq_interp_fD_A_2}
        \abs{\fD^\l A_{2, p, k}} \leq 
        \begin{cases}
            C e^{-\frac{t}{2} \br{4 + \alpha -\frac{(2+\alpha)\l}{4+\alpha}}}, & \forall 0 \leq \l \leq 4, \\
            C e^{-\frac{t}{2} \br{4 +  \frac{\l(2-\l) + \alpha^2}{\l+\alpha}} }, & \forall 4 \leq \l \leq 2k+ 6.
        \end{cases}
    \end{equation}
	\end{lemma}
	\begin{proof}
		By \eqref{eq_expa_A_i_holder}, $[\fD^4 A_{2, p, k}]_\alpha \leq C e^{-t}$, and $[\fD^\l A_{2, p, k}]_\alpha \leq C e^{\frac{t}{2} \br{\l-6}}$ for $4 \leq \l \leq 2k+6$. Interpolate between each case and $\abs{A_{2, p, k}} \leq C e^{-(4+\alpha) \frac{t}{2}}$ by \eqref{eq_expa_A_i_infty}  as in Lemma \ref{lem_interp_fD_A_1} to derive \eqref{eq_interp_fD_A_2}.
	\end{proof}

	We can use \Cref{lem_interp_fD_A_1,lem_interp_fD_A_2} to obtain higher-order estimates on $\gamma_{i,k}$.
	
	\begin{proposition}\label{prop_bound_D_gamma_1}
		We have
		\begin{equation}
			\abs{\fD^\l \gamma_{1, k}} \leq C e^{(\l-2) \frac{t}{2}}, \quad \forall 0 \leq \l \leq 2.
		\end{equation}
	\end{proposition}
	\begin{proof}
		See \cite[Proposition 5.4]{HLT}. The idea is to apply \eqref{eq_fD^q_gamma_i_k} and bound each term in its RHS.
	\end{proof}
	
	\begin{proposition}\label{prop_bound_D_dot_gamma_1_2}
		For both $i = 1, 2$, we have
		\begin{equation}
			\abs{\D \dot{\gamma}_{i, k}} \leq C.
		\end{equation}
	\end{proposition}
	\begin{proof}
		Take $t$-derivative of $\D \gamma_{i, k}$ given by \eqref{eq_D^q_gamma_i_k}, and estimate using Lemma \ref{lem_interp_fD_A_1} and Lemma \ref{lem_interp_fD_A_2}:
		\begin{equation}
			\begin{aligned}
				& \abs{\D \dgamma_{i, k}} \\
				& \leq C \sum_{p=1}^{N_{i,k}}\sum_{\iota=0}^{2k}\sum_{r=\lceil\frac{\iota}{2}\rceil}^{k}\sum_{s=0}^{2}\sum_{i_1+i_2=s+1}
				e^{-rt}  \abs{(\D^{2-s}J)} \cdot \abs{\D^{i_1}\Phi_{\iota,r}(G_{1,p,k})} \br{\abs{\D^{i_2+\iota} A_{i,p,k}}
					+ \abs{\D^{i_2+\iota} \dot{A}_{i,p,k}}} \\
				& \leq C \sum_{p=1}^{N_{i,k}}\sum_{\iota=0}^{2k}\sum_{r=\lceil\frac{\iota}{2}\rceil}^{k}\sum_{i_2=0}^{3}
				e^{\thalf \br{-2r + 3-i_2}} \br{\abs{\fD^{i_2+\iota} A_{i,p,k}}
					+ \abs{\fD^{i_2+\iota+2} {A}_{i,p,k}}} \\
				& \leq C,
			\end{aligned}
		\end{equation}
		as $\abs{\D^\l J} \leq C e^{\l \frac{t}{2}}$ and $\abs{\D^\l \Phi_{\iota,r}(G_{1,p,k})} \leq C e^{\l \frac{t}{2}}$ for all $\l \in \N$.
	\end{proof}

	We are now ready to prove \Cref{thm_D^1_ric}.
	\begin{proof}[Proof of \Cref{thm_D^1_ric}]
		Combine the \KR flow equation \eqref{eq_krf} with asymptotic expansion \eqref{eq_expa_j=1,2}, to get
		\begin{equation}
			\Ric(\ob)
				= - \ocan - \gamma_0 - \gamma_{1, k} - \eta_{1, k} - \dot{\gamma}_0 - \dot{\gamma}_{1, k} -\dot{ \gamma}_{2, k} - \dot{\eta}_{2, k}.
		\end{equation}
		We bound the $\D$-derivative of each term in the RHS above:
		\begin{enumerate}
			\item $\abs{\D \ocan} \leq C$ since $\ocan$ (and hence $\D \ocan$) lives on the base.
			
			\item $\abs{\D \gamma_0} = o(1)$ by \eqref{eq_expa_gamma_0_infty}.
			
			\item $\abs{\D \gamma_{1, k}} \leq C e^{-\frac{t}{2}}$ by Proposition \ref{prop_bound_D_gamma_1}
			
			\item $\abs{\D \eta_{1, k}} \leq C e^{-(1+\alpha)\frac{t}{2}}$ by \eqref{eq_expa_eta_infty}.
			
			\item $\abs{\D \dot{\gamma}_0} \leq \abs{\fD^3 \gamma_0 } = o(1)$ by \eqref{eq_expa_gamma_0_infty}.
			
			\item $\abs{\D \dot{\gamma}_{1, k}} \leq C$, $\abs{\D \dot{\gamma}_{2, k}} \leq C$ by Proposition \ref{prop_bound_D_dot_gamma_1_2}. 
			
			\item $\abs{\D \dot{\eta}_{2, k}} \leq \abs{\fD^3 {\eta}_{2, k}} \leq C e^{-\br{1+\alpha} \frac{t}{2} }$ by \eqref{eq_expa_eta_infty}.
			
		\end{enumerate}
		Therefore, $\abs{\D \Ric(\ob(t))}_{g(t)} \leq C$. By uniform equivalence between $g(t)$ and $g^\bullet(t)$ (see Lemma \ref{lem_metric_equiv}), we deduce that $\abs{\D \Ric(\ob(t))}_{g^\bullet(t)} \leq C$, which completes the proof.
	\end{proof}

	\subsection{Proof of \texorpdfstring{\Cref{prop_D_nabla_b_Ric}}{Theorem 3.1}}
	Essentially our task is to compare $\nabla^\bullet$ with $\D$. Since the complex structure $J$ on $B \times Y$ is not necessarily a product, we are forced to analyze the symmetric 2-tensor $g^\bullet(t)$ instead of the 2-form $\ob(t)$. To streamline the process, define $T^J :=T \circ \br{\id \otimes J}$, for any covariant 2-tensor $T$ on $B \times Y$. Then we can write
	\begin{equation}\label{eq_g_bullet_j=1}
		g^\bullet = \br{\ob}^J = (1-e^{-t}) \gcan + e^{-t} g_F + \gamma_0^J + \gamma_{1, k}^J + \eta_{1, k}^J,
	\end{equation}
	where $\gcan := \ocan^J$ and $g_F := \omega_F^J$.
	Also, as in \cite{HLT}, we use $\bb$ and $\ff$ to denote the base and fiber components of a tensor on $B \times Y$, respectively, according to the product splitting $T(B \times Y) = TB \oplus TY$.
	
	First, observe the following product rule for $\D T^J$.
	\begin{lemma}\label{lem_prod_rule_DJ}
		For any covariant 2-tensor $T$ on $B \times Y$, we have in coordinates
		\begin{equation}\label{eq_prod_rule_DJ_coord}
			\br{\D T^J}_{ijk} = \br{\D T}_{ija} J_k^a + T_{ja} \br{\D J}_{ik}^a.
		\end{equation}
		Schematically,
		\begin{equation}\label{eq_prod_rule_DJ}
			\D^\l T^J = \sum_{m=0}^\l \D^m T \circledast \D^{\l-m} J, \quad \forall \l \geq 0.
		\end{equation}
	\end{lemma}
	\begin{proof}
		By definition,
		\begin{equation}
			T^J = T \circledast J,
		\end{equation}
		written in coordinates as
		\begin{equation}
			\br{T^J}_{jk} = T_{ja} J_{k}^a.
		\end{equation}
		Thus the product rule for $\D$ applies to yield \eqref{eq_prod_rule_DJ_coord} and \eqref{eq_prod_rule_DJ}.
	\end{proof}
	
	We can then show the following estimates.
	
	\begin{lemma}\label{lem_D_gcan}
		We have
		\begin{equation}\label{eq_D_gcan}
			\abs{\D^\l \gcan} \leq C, \quad \forall \l \geq 0.
		\end{equation}
	\end{lemma}
	\begin{proof}
		Recall that $f: B \times Y \to B$  is $(J, J_{\C^m})$-holomorphic, so
		\begin{equation}
			\gcan = (f^*\ocan) \circ \br{\id \otimes J} = f^* \br{\ocan \circ (\id \otimes J_{\C^m})}
		\end{equation} 
		lives on the base. Hence $\D^\l \gcan$ is a $t$-independent tensor on the base for all $\l \geq 0$, and \eqref{eq_D_gcan} follows.
	\end{proof}

	\begin{lemma}\label{lem_D_gamma_0_J}
		We have
		\begin{equation}\label{eq_D_gamma_0_J}
			\abs{\D^\l \gamma_0^J} = o(1), \quad \forall 0 \leq \l \leq 4.
		\end{equation}
	\end{lemma}
	\begin{proof}
		Recall $\gamma_0 = \ddbar \br{f^*\underline{\varphi}}$ from \eqref{eq_gamma_formula}, so
		\begin{equation}
			\gamma_{0}^J = f^*\br{\ddbar \underline{\varphi} \circ (\id \otimes J_{\C^m})},
		\end{equation}
		and $\D^\l \gamma_0$, $\D^\l \gamma_0^J$ live on the base for all $\l \geq 0$. Therefore, by Lemma \ref{lem_prod_rule_DJ},
		\begin{equation}
			\begin{aligned}
				\abs{\D \gamma_0^J}
				& = \abs{\br{\D \gamma_0} \circledast J_{\bb}^{\bb} + {\gamma_0} \circledast \br{\D J}_{\bb \bb}^{\bb}} \\
				& \leq C \br{\abs{\D \gamma_0} + \abs{\gamma_0}},
			\end{aligned}
		\end{equation}
		and more generally,
		\begin{equation}
			\abs{\D^\l \gamma_0^J} \leq C \sum_{m=0}^\l \abs{\D^m \gamma_0}.
		\end{equation}
		We then use \eqref{eq_expa_gamma_0_infty} to conclude \eqref{eq_D_gamma_0_J}.
	\end{proof}
	
	\begin{lemma}\label{lem_D_gamma_1_J}
		We have
		\begin{equation}
			\abs{\D^\l \gamma_{1, k}^J} \leq C e^{(\l-2)\thalf}, \quad \forall 0 \leq \l \leq 2.
		\end{equation}
	\end{lemma}
	\begin{proof}
		This follows immediately from Proposition \ref{prop_bound_D_gamma_1} and Lemma \ref{lem_prod_rule_DJ}.
	\end{proof}
	
	\begin{lemma}\label{lem_D_eta_1_J}
		We have
		\begin{equation}
			\abs{\D^\l \eta_{1, k}^J} \leq C e^{(\l-2-\alpha)\thalf}, \quad \forall 0 \leq \l \leq 2.
		\end{equation}
	\end{lemma}
	\begin{proof}
		This follows immediately from \eqref{eq_expa_eta_infty} and Lemma \ref{lem_prod_rule_DJ}.
	\end{proof}
	
	\begin{lemma}\label{lem_D_gF}
		We have
		\begin{equation}
			\abs{\D g_F} \leq C e^t.
		\end{equation}
	\end{lemma}
	\begin{proof}
		Observe that 
		\begin{equation}\label{eq_D_gF_fff}
			\br{\D g_F}_{\ff \ff \ff} = 0.
		\end{equation}
		To see this, recall that for all $z \in B$, the biholomorphism between $X_z$ and  $\{z\} \times Y$ (equipped with complex structure $J_z = J|_{\{z\} \times Y}$) identifies the Ricci-flat metric $\omega_F|_{X_z}$ with $g_{Y, z}$. Thus
		\begin{equation}
			g_F|_{\{z\} \times Y} = g_{Y, z}.
		\end{equation}
		On the fiber $\{z\} \times Y$, we then have
		\begin{equation}\label{eq_D_gF_nabla}
			\br{\D g_F}_{\ff \ff \ff}(z, y) = \br{\nabla^z g_F}_{\ff \ff \ff}(z, y) = \br{\nabla^{g_{Y, z}} g_{Y, z}} (y) = 0, \quad \forall y \in Y.
		\end{equation}
		Thus \eqref{eq_D_gF_fff} holds.
		
		Since $g(t) = g_{\C^m} + e^{-t} g_{Y, 0}$ is static in base and shrinks the size of fiber at rate $e^{-\frac{t}{2}}$, the $t$-independent covariant 3-tensor $\D g_F$ with \eqref{eq_D_gF_fff} must satisfy $\abs{\D g_F}_{g(t)} \leq C e^t$.
	\end{proof}

	\begin{proposition}\label{prop_bound_D_gb}
		We have
		\begin{equation}\label{eq_bound_D_gb}
			\abs{\D \gb} \leq C,
		\end{equation}
		\begin{equation}\label{eq_bound_D_gb-1}
			\abs{\D \br{\gb}^{-1}} \leq C.
		\end{equation}
	\end{proposition}
	\begin{proof}
		Decompose $g^\bullet$ as in \eqref{eq_g_bullet_j=1} and use \Cref{lem_D_gcan,lem_D_gF,lem_D_gamma_0_J,lem_D_gamma_1_J,lem_D_eta_1_J} to conclude \eqref{eq_bound_D_gb}. Note that $\D \br{\gb}^{-1} = \br{\gb}^{-1} \circledast \br{\gb}^{-1} \circledast \D \gb$ and hence \eqref{eq_bound_D_gb-1} follows.
	\end{proof}

	We are now ready to prove \Cref{prop_D_nabla_b_Ric}.
	\begin{proof}[Proof of \Cref{prop_D_nabla_b_Ric}]
		Let $A$ denote the difference tensor between $\nabla^\bullet$ and $\D$. Then $A(x) = A^{\pr_B(x)}(x)$, where $A^z$ for $z \in B$ denote the difference tensor between $\nabla^\bullet$ and $\nabla^z$. In coordinates, we have
		\begin{equation}\label{eq_A_z_formula}
			\br{A^z}_{ij}^k = \br{\Gamma^z}_{ij}^k - \br{\Gamma^\bullet}_{ij}^k = -\frac{1}{2} g^{\bullet k\l} \br{\nabla_i^{z} g^\bullet_{j\l} + \nabla_j^{z} g^\bullet_{i\l} -\nabla_\l^{z} g^\bullet_{ij} },
		\end{equation}
		so that schematically
		\begin{equation}\label{eq_nabla_b-D_formula}
			A = \br{\gb}^{-1} \circledast \D g^\bullet, \quad \br{\nabla^\bullet - \D} T = {A \circledast T}, 			
		\end{equation}
		for all tensors $T$ on $B \times Y$.
		
		Then by Lemma \ref{lem_metric_equiv} and Proposition \ref{prop_bound_D_gb}, 
		\begin{equation}\label{eq_bound_A_g}
			\abs{A} \leq C \abs{\br{\gb}^{-1}} \cdot \abs{\D \gb} \leq C,
		\end{equation}
		and equivalently
		\begin{equation}\label{eq_bound_A_gb}
			\abs{A}_{\gb} \leq C.
		\end{equation}
		Finally, combine \eqref{eq_bound_ric}, \eqref{eq_nabla_b-D_formula}, and \eqref{eq_bound_A_gb}, to get
		\begin{equation}
				\abs{\br{\nabla^\bullet - \D} \Ric(\omega^\bullet)}_{\gb}
				\leq C \abs{A}_{\gb} \cdot \abs{\Ric(\ob)}_{\gb} 
				 \leq C.
		\end{equation}
		This completes the proof of \Cref{prop_D_nabla_b_Ric}.
	\end{proof}

\section{Second-Order Curvature Estimates}\label{section_order_2}

In this section we prove \Cref{thm_D^2_ric}, which is a second-order estimate on the Ricci curvature in the parabolic sense, and use it to derive \Cref{thm_nabla^2_scalar,thm_Delta_ric}. We work in the same local framework on the product space $B \times Y$ with the same simplification of notations as in \Cref{section_order_1}. 

Similar to \Cref{section_order_1}, after proving \Cref{thm_D^2_ric}, we will calculate the difference tensor $\br{\nabla^{\ob(t), 2} - \D^2} \Ric(\ob(t))$, in which we will identify the obstruction to bounding $\nabla^{\ob(t), 2} \Ric(\ob(t))$ uniformly. Nevertheless, if we trace the difference 4-tensor with respect to the flow metric $\ob(t)$, our estimates on each component tensor of $\Ric(\ob(t))$ and $\gb(t)^{-1}$ in their base-fiber decomposition will enable us to conclude \Cref{thm_nabla^2_scalar}. The same idea of base-fiber analysis on tensors, combined with the evolution equation for the Ricci curvature along the Ricci flow, leads to \Cref{thm_Delta_ric}.

	\subsection{Proof of \texorpdfstring{\Cref{thm_D^2_ric}}{Theorem 1.5}}
	We apply Theorem \ref{thm_asymp_expa} with any fixed even integer $k \geq 4$ and $j = 1, 2$. Up to shrinking $B$, we can write on $B \times Y \times \infint$
	\begin{equation}\label{eq_expa_j=2}
		\ob = \on + \gamma_0 + \gamma_{1, k} +\eta_{1, k} = \on + \gamma_0 + \gamma_{1, k} + \gamma_{2, k} + \eta_{2, k}.
	\end{equation}
	
	Let $g_X = g(0)$ denote a $t$-independent product metric on $B \times Y$.
	We first establish the following estimates.
	\begin{lemma}\label{lem_bound_gamma_i_k_g_X}
		For both $i = 1, 2$, we have 
		\begin{equation}\label{eq_bound_gamma_i_k_g_X}
			\abs{\gamma_{i, k}}_{g_X} \leq C e^{-3\frac{t}{2}},  
		\end{equation}
		\begin{equation}\label{eq_bound_dot_gamma_i_k_g_X}
			\abs{\dot{\gamma}_{i, k}}_{g_X} \leq C e^{-t},  
		\end{equation}
		\begin{equation}\label{eq_bound_gamma_i_k_ff_g_X}
			\abs{\br{\gamma_{i, k}}_{\ff \ff}}_{g_X} \leq C e^{-2t},
		\end{equation}
		\begin{equation}\label{eq_bound_dot_gamma_i_k_ff_g_X}
			\abs{\br{\dot{\gamma}_{i, k}}_{\ff \ff}}_{g_X} \leq C e^{-3\thalf},
		\end{equation}
		\begin{equation}\label{eq_bound_dot_gamma_i_k_ff_g_X_dom}
			\abs{\br{\dot{\gamma}_{i, k}}_{\ff \ff} - \sum_{p=1}^{N_{i, k}} \dot{A}_{i, p, k} i \partial_{\ff} \overline{\partial}_{\ff} \br{\Delta^{\omega_F|_{\{\cdot\} \times Y}}}^{-1} G_{i, p, k}}_{g_X} \leq C {e^{-2t}}.
		\end{equation}
	\end{lemma}
	\begin{proof}
		Using \eqref{eq_D^q_gamma_i_k} and its $t$-derivative, we have
		\begin{equation}
			\begin{aligned}
				\abs{\gamma_{i,k}}_{g_X}
				& \leq C \sum_{p=1}^{N_{i,k}}\sum_{\iota=0}^{2k}\sum_{r=\lceil\frac{\iota}{2}\rceil}^{k}\sum_{i_2=0}^{2} e^{-r t} \abs{\D^{i_2+\iota} A_{i,p,k}},
			\end{aligned}
		\end{equation}
		\begin{equation}
			\begin{aligned}
				\abs{\dgamma_{i,k}}_{g_X}
				& \leq C \sum_{p=1}^{N_{i,k}}\sum_{\iota=0}^{2k}\sum_{r=\lceil\frac{\iota}{2}\rceil}^{k}\sum_{i_2=0}^{2} e^{-r t} \br{\abs{\D^{i_2+\iota} A_{i,p,k}} + \abs{\D^{i_2+\iota} \dot{A}_{i,p,k}}}.
			\end{aligned}
		\end{equation}
		We can then apply Lemma \ref{lem_interp_fD_A_1} when $i = 1$ and Lemma \ref{lem_interp_fD_A_2} when $i = 2$ to get \eqref{eq_bound_gamma_i_k_g_X} and \eqref{eq_bound_dot_gamma_i_k_g_X}. Similarly, using \eqref{eq_gamma_i_k}, we have
		\begin{equation}\label{eq_gamma_i_k_ff}
			\br{\gamma_{i,k}}_{\ff \ff} = \sum_{p=1}^{N_{i,k}}\sum_{\iota=0}^{2k}\sum_{r=\lceil\frac{\iota}{2}\rceil}^{k}e^{-r t} i\de_{\ff} \dbar_{\ff} \Phi_{\iota,r}(G_{i,p,k})\circledast\D^\iota A_{i,p,k},
		\end{equation}
		so
		\begin{equation}
			\abs{\br{\gamma_{i,k}}_{\ff \ff}}_{g_X} \leq C  \sum_{p=1}^{N_{i,k}}\sum_{\iota=0}^{2k}\sum_{r=\lceil\frac{\iota}{2}\rceil}^{k}e^{-r t} \abs{\D^\iota A_{i,p,k}} \leq C e^{-2t}.
		\end{equation}
		By \eqref{eq_Phi_00},
		\begin{equation}
			\begin{aligned}
				\br{\dgamma_{i,k}}_{\ff \ff} 
				& = \sum_{p=1}^{N_{i, k}} \dot{A}_{i, p, k}  i\partial_{\ff} \overline{\partial}_{\ff} \br{\Delta^{\omega_F|_{\{\cdot\} \times Y}}}^{-1} G_{i, p, k} + \sum_{p=1}^{N_{i,k}}\sum_{(\iota, r) \neq (0, 0)}e^{-r t} \de_{\ff} \dbar_{\ff} \Phi_{\iota,r}(G_{i,p,k})\circledast\D^\iota \dot{A}_{i,p,k} \\
				& \quad + \sum_{p=1}^{N_{i,k}}\sum_{\iota=0}^{2k}\sum_{r=\lceil\frac{\iota}{2}\rceil}^{k} e^{-r t} i\de_{\ff} \dbar_{\ff} \Phi_{\iota,r}(G_{i,p,k})\circledast\D^\iota A_{i,p,k},
			\end{aligned}
		\end{equation}
		and we can estimate as above:
		\begin{equation}
			\begin{aligned}
				&\abs{\br{\dot{\gamma}_{i,k}}_{\ff \ff} - \sum_{p=1}^{N_{i, k}} \dot{A}_{i, p, k}  i\partial_{\ff} \overline{\partial}_{\ff} \br{\Delta^{\omega_F|_{\{\cdot\} \times Y}}}^{-1} G_{i, p, k}}_{g_X} \\
				& \leq C  \sum_{p=1}^{N_{i,k}}\sum_{(\iota, r) \neq (0, 0)}e^{-r t} {\abs{\D^\iota \dot{A}_{i,p,k}}} 
				+ C\sum_{p=1}^{N_{i,k}}\sum_{\iota=0}^{2k}\sum_{r=\lceil\frac{\iota}{2}\rceil}^{k}e^{-r t} { \abs{\D^\iota {A}_{i,p,k}}} \\
				& \leq C e^{-2t},			
			\end{aligned}
		\end{equation}
		and
		\begin{equation}
			\abs{\dot{A}_{i, p, k}  i\partial_{\ff} \overline{\partial}_{\ff} \br{\Delta^{\omega_F|_{\{\cdot\} \times Y}}}^{-1} G_{i, p, k}}_{g_X} \leq C \abs{\fD^2 A_{i, p, k}} \leq C e^{-3 \thalf}.
		\end{equation}
		The proof is thus complete.
	\end{proof}
	
	\begin{lemma}\label{lem_bound_eta_2}
		We have 
		\begin{equation}\label{eq_bound_eta_2_g_X}
			\abs{\eta_{2, k}}_{g_X} \leq C e^{-(4+\alpha)\frac{t}{2}},  
		\end{equation}
		\begin{equation}\label{eq_bound_dot_eta_2_g_X}
			\abs{\deta_{2, k}}_{g_X} \leq C e^{-(2+\alpha) \thalf},  
		\end{equation}
		\begin{equation}\label{eq_bound_eta_2_ff_g_X}
			\abs{\br{\eta_{2, k}}_{\ff \ff}}_{g_X} \leq C e^{-(6+\alpha) \thalf}.
		\end{equation}
		\begin{equation}\label{eq_bound_dot_eta_2_ff_g_X}
			\abs{\br{\deta_{2, k}}_{\ff \ff}}_{g_X} \leq C e^{-(4+\alpha)\thalf}.
		\end{equation}
	\end{lemma}
	\begin{proof}
		Note that for any covariant $\l$-tensor $T$ on $B \times Y$,
		\begin{equation}
			\abs{T_{\ff \dots \ff}}_{g_X} = e^{-\l \thalf} \abs{T_{\ff \dots \ff}}_{g(t)} \leq e^{-\l \thalf} \abs{T}_{g(t)}.
		\end{equation} 
		Hence all the estimates above follow from \eqref{eq_expa_eta_infty}.
	\end{proof}

	\begin{lemma}\label{prop_davg_phi}
		We have
		\begin{equation}\label{eq_davg_phi}
			\abs{\davg{{\varphi}}} = o\br{e^{-t}},
		\end{equation}
		\begin{equation}\label{eq_davg_dot_phi}
			\abs{\davg{\dot{\varphi}}} = o\br{e^{-t}},
		\end{equation}
		\begin{equation}\label{eq_davg_ddot_phi}
			\abs{\davg{\ddot{\varphi}}} = o\br{e^{-t}}.
		\end{equation}
	\end{lemma}
	\begin{proof}
		\eqref{eq_davg_phi} and \eqref{eq_davg_dot_phi} were proved in \cite[Proposition 5.1]{HLT}. We adapt the argument to show \eqref{eq_davg_ddot_phi}. By Theorem \ref{thm_asymp_expa} we can write
		\begin{equation}\label{eq_phi_fiber_order_2}
			\varphi - \underline{\varphi}
			= \sum_{p=1}^{N_{1, k}} \mathfrak{G}_{t, k}(A_{1, p, k}, G_{1, p, k})
			+ \sum_{p=1}^{N_{2, k}} \mathfrak{G}_{t, k}(A_{2, p, k}, G_{2, p, k}) + \psi_{2, k}
		\end{equation}	
		where taking $\ddbar$ of terms on the RHS yields $\gamma_{1, k}$, $\gamma_{2, k}$, $\eta_{2, k}$, respectively. For $i = 1, 2$, $1 \leq p \leq N_{i, k}$, by \eqref{eq_Green_t_k},
		\begin{equation}
			\abs{\partial_t^2 \br{\mathfrak{G}_{t, k}(A_{i, p, k}, G_{i, p, k})}} \leq C \sum_{\iota = 0}^{2k} \sum_{r = \ceil{\frac{\iota}{2}}}^k e^{-rt} \br{\abs{\D^\iota A_{i, p, k}} + \abs{\D^\iota \dot{A}_{i, p, k}} + \abs{\D^\iota \ddot{A}_{i, p, k}}}.
		\end{equation}
		We can then apply Lemma \ref{lem_interp_fD_A_1} when $i = 1$ and Lemma \ref{lem_interp_fD_A_2} when $i = 2$ to get 
		\begin{equation}\label{eq_dt_2_Green}
			\abs{\partial_t^2 \br{\mathfrak{G}_{t, k}(A_{i, p, k}, G_{i, p, k})}} = o\br{e^{-t}}.
		\end{equation}

		To handle $\psi_{2, k}$, recall from \eqref{eq_expa_eta_infty} that
		\begin{equation}
			\abs{ \ddbar \ddot{\psi}_{2, k}}_{g(t)} = \abs{\ddot{\eta}_{2, k}}_{g(t)} \leq C e^{-\alpha \frac{t}{2}}.
		\end{equation}
		Restricting $\ddot{\psi}_{2, k}$ to each fiber $\{z\} \times Y$, we see
		\begin{equation}
			\abs{\ddbar \ddot{\psi}_{2, k}|_{\{z\} \times Y}}_{g_{Y, z}} \leq C e^{-(2+\alpha) \thalf}.
		\end{equation}
		Since $\ddot{\psi}_{2, k}$ has fiberwise average zero, we can apply Moser's iteration on each fiber $\{z\} \times Y$, with metric $g_{Y, z}$ varying smoothly along $z \in B$, to get
		\begin{equation}\label{eq_ddot_psi_2}
			\abs{\ddot{\psi}_{2, k}} \leq C e^{-(2+\alpha) \frac{t}{2}},
		\end{equation}
		where the constant $C$ is uniform in $B \times Y$, up to shrinking $B$. We can thus derive \eqref{eq_davg_ddot_phi} using \eqref{eq_phi_fiber_order_2}, \eqref{eq_dt_2_Green}, and \eqref{eq_ddot_psi_2}.
	\end{proof}
	
	\begin{proposition}\label{prop_bound_dot_A_i_leq2}
		We have
		\begin{equation}\label{eq_bound_dot_A_i}
			\abs{\dot{A}_{i, p, k}} \leq C e^{-2t}, \quad \forall i = 1, 2, \quad 1 \leq p \leq N_{i, k}.
		\end{equation}
	\end{proposition}
	
	\begin{proof}

		Using \eqref{eq_expa_j=2}, we write the parabolic \MA equation \eqref{eq_krf_MA} for the \KR flow as
		\begin{equation}
			e^{\varphi + \dot{\varphi}} \omega_{\can}^m \wedge \omega_F^n
			= \frac{e^{nt}}{\binom{m+n}{n}} \br{(1-e^{-t}) \omega_{\can} + e^{-t} \omega_F + \gamma_0 + \gamma_{1, k} + \gamma_{2, k} + \eta_{2, k}}^{m+n}.
		\end{equation}
		We take its $t$-derivative and use \eqref{eq_expa_gamma_0_infty}, Lemma \ref{lem_bound_gamma_i_k_g_X}, Lemma \ref{lem_bound_eta_2} to get
		\begin{equation}
			\begin{aligned}
				& \br{\dot{\varphi} + \ddot{\varphi} - n} e^{\varphi + \dot{\varphi}} \omega_{\can}^m \wedge \omega_F^n \\
				& = \frac{(m+n) e^{nt}}{\binom{m+n}{n}} \br{(1-e^{-t}) \omega_{\can} + \gamma_0 + e^{-t} \br{\omega_F + e^t \gamma_{1, k} + e^t \gamma_{2, k} + e^t \eta_{2, k}}}^{m+n-1} \\
				& \quad \wedge \br{e^{-t} \omega_{\can} + \dot{\gamma}_0 + e^{-t} \br{-\omega_F + e^t \dot{\gamma}_{1, k} + e^t \dot{\gamma}_{2, k} + e^t \dot{\eta}_{2, k}}} \\
				& =  n \br{(1-e^{-t}) \omega_{\can} + \gamma_0}^m \br{\omega_F}_{\ff \ff}^{n-1} 
				\br{-\omega_F + e^t \dot{\gamma}_{1, k} + e^t \dot{\gamma}_{2, k} }_{\ff \ff} \\
				& \quad + m \br{(1-e^{-t}) \omega_{\can} + \gamma_0}^{m-1}
				\dot{\gamma}_0
				\br{\omega_F}_{\ff \ff}^{n} \\
				& \quad + O_{g_X} \br{e^{-t}}. 
			\end{aligned}
		\end{equation}
		Divide both sides of the equality above by $e^t \omega_{\can}^m \wedge \omega_F^n $ to get
		\begin{equation}\label{eq_flow_dot}
			\begin{aligned}
				&\br{\dot{\varphi} + \ddot{\varphi} - n}  e^{-t} e^{\varphi + \dot{\varphi}} \\
				&= \br{1+o(1)_{\text{base}}} \br{-n e^{-t} + \tr{\omega_F|_{\{\cdot\} \times Y}}{\br{ \dot{\gamma}_{1, k} + \dot{\gamma}_{2, k}}_{\ff \ff}} } +o(1)_{\text{base}} + O_{g_X}(e^{-2t}).
			\end{aligned}
		\end{equation}
		
		Next we subtract from each side of \eqref{eq_flow_dot} their fiberwise average. For RHS, this will indeed remove all terms that live on the base, and $\tr{\omega_F|_{\{\cdot\} \times Y}}{\br{ \dot{\gamma}_{1, k} + \dot{\gamma}_{2, k}}_{\ff \ff}}$ has fiberwise average zero since $\dot{\gamma}_{i, k}$ are $\de\dbar$-exact.  For LHS, note that for arbitrary functions $f, g$ on $B \times Y$, we have
		\begin{equation}\label{eq_prod_avg_diff}
			fg - \underline{fg} = \br{f - \underline{f}} g + \underline{f} \br{g-\underline{g}} - \underline{\br{f-\underline{f}} g}.
		\end{equation}
		Plug in $f = \dot{\varphi} + \ddot{\varphi} - n$ and $g = e^{\varphi + \dot{\varphi}}$. Since $\varphi, \dot{\varphi}, \ddot{\varphi}$ are uniformly bounded (thanks to Lemma \ref{lem_known_est}), we can use \Cref{prop_davg_phi} and the Taylor expansion of the exponential to get
		\begin{align*}
			e^{-t}\abs{\br{\dot{\varphi} + \ddot{\varphi} - n} e^{\varphi + \dot{\varphi}} - \underline{\br{\dot{\varphi} + \ddot{\varphi} - n} e^{\varphi + \dot{\varphi}}}}
			& \leq C e^{-2t}.
		\end{align*}
		Therefore, \eqref{eq_flow_dot} yields
		\begin{equation}
			\tr{\omega_F|_{\{\cdot\} \times Y}}{\br{ \dot{\gamma}_{1, k} + \dot{\gamma}_{2, k}}_{\ff \ff}} = O \br{e^{-2t}}.
		\end{equation}
		Combined with \eqref{eq_bound_dot_gamma_i_k_ff_g_X_dom}, we have
		\begin{equation}\label{eq_bound_AG}
			\sum_{p=1}^{N_{1, k}} \dot{A}_{1, p, k} G_{1, p, k} + \sum_{p=1}^{N_{2, k}} \dot{A}_{2, p, k} G_{2, p, k} = O \br{e^{-2t}}.
		\end{equation}
		Recall from Theorem \ref{thm_asymp_expa} that $\{G_{i, p, k} \mid {1 \leq i \leq 2, 1 \leq p \leq N_{i,k}}\}$ are fiberwise $L^2$ orthonormal, and $\{A_{i, p, k} \mid {1 \leq i \leq 2, 1 \leq p \leq N_{i,k}}\}$ are functions on the base. We can thus take fiberwise inner product of \eqref{eq_bound_AG} with each $G_{i, p, k}$ to get \eqref{eq_bound_dot_A_i}.
	\end{proof}

	We can therefore improve \Cref{lem_interp_fD_A_1,lem_interp_fD_A_2}.
	\begin{lemma}\label{lem_interp_fD_dot_A_1_2}
		For all $i = 1, 2$, $1 \leq p \leq N_{i, k}$, 
		\begin{equation}\label{eq_interp_fD_dot_A_1_2}
			\abs{\fD^\l \dot{A}_{i, p, k}} \leq 
			\begin{cases}
				C e^{-\frac{t}{2} \br{4-\frac{2\l}{2+\alpha}}}, & \forall 0 \leq \l \leq 2, \\
				C e^{-\frac{t}{2} \br{4 - \frac{\l^2}{\l+\alpha}}}, & \forall 2 \leq \l \leq 2k+ 4.
			\end{cases}
		\end{equation}
	\end{lemma}
	\begin{proof}
		The case $\l = 0$ is exactly Proposition \ref{prop_bound_dot_A_i_leq2}. By \eqref{eq_expa_A_i_holder}, 
		\begin{equation}
			\sqbr{\fD^\l \dot{A}_{i, p, k}}_\alpha \leq \sqbr{\fD^{\l+2} {A}_{i, p, k}}_\alpha \leq Ce^{\frac{t}{2}(\l-4)}, \quad \forall 2 \leq \l \leq 2k+4,
		\end{equation}
		between which and $\abs{\dot{A}_{i,p,k}} \leq C e^{-2t}$ we can interpolate to derive \eqref{eq_interp_fD_dot_A_1_2}.
	\end{proof}
	We can then improve \Cref{prop_bound_D_dot_gamma_1_2}.
	
	\begin{proposition}\label{prop_bound_D2_dot_gamma_1_2}
		For both $i = 1, 2$,
		\begin{equation}
			\abs{\fD^\l \dgamma_{i, k}} \leq C e^{(\l-2)\thalf}, \quad \forall 0 \leq \l \leq 2.
		\end{equation}
	\end{proposition}
	\begin{proof}
		Take $t$-derivative of $\fD^\l \gamma_{i, k}$ given by \eqref{eq_fD^q_gamma_i_k}, and estimate using Lemma \ref{lem_interp_fD_A_1}, Lemma \ref{lem_interp_fD_A_2}, Lemma \ref{lem_interp_fD_dot_A_1_2}:
		\begin{equation}
			\begin{aligned}
				\abs{\fD^\l \dgamma_{i, k}}
				& \leq C \sum_{p=1}^{N_{i,k}}\sum_{\iota=0}^{2k}\sum_{r=\lceil\frac{\iota}{2}\rceil}^{k}\sum_{s=0}^{\l+1}\sum_{i_1+i_2=s+1}
				e^{\thalf \br{-2r + \l+1-s + i_1}} \br{\abs{\fD^{i_2+\iota} A_{i,p,k}}
				+ \abs{\fD^{i_2+\iota} \dot{A}_{i,p,k}}} \\
				& \leq C \sum_{p=1}^{N_{i,k}}\sum_{\iota=0}^{2k}\sum_{r=\lceil\frac{\iota}{2}\rceil}^{k}\sum_{i_2=0}^{\l+2}
				e^{\thalf \br{-2r + 2+\l-i_2}} \br{\abs{\fD^{i_2+\iota} A_{i,p,k}}
					+ \abs{\fD^{i_2+\iota} \dot{A}_{i,p,k}}} \\
				& \leq C e^{(\l-2)\thalf}.
			\end{aligned}
		\end{equation}
	\end{proof}

	We are now ready to prove \Cref{thm_D^2_ric}.
	\begin{proof}[Proof of \Cref{thm_D^2_ric}]
		Combine the \KR flow equation \eqref{eq_krf} with the expansion \eqref{eq_expa_j=2}, to get
		\begin{equation}
			\Ric(\ob)
			= - \ocan - \gamma_0 - \gamma_{1, k} - \eta_{1, k} - \dot{\gamma}_0 - \dot{\gamma}_{1, k} -\dot{ \gamma}_{2, k} - \dot{\eta}_{2, k}.
		\end{equation}
		We bound the $\fD^2$-derivative of each term in the RHS above:
		\begin{enumerate}
			\item $\abs{\fD^2 \ocan} \leq C$ since $\ocan$ lives on the base and is $t$-independent.
			
			\item $\abs{\fD^2 \gamma_0} = o(1)$ by \eqref{eq_expa_gamma_0_infty}.
			
			\item $\abs{\fD^2 \gamma_{1, k}} \leq C $ by Proposition \ref{prop_bound_D_gamma_1}
			
			\item $\abs{\fD^2 \eta_{1, k}} \leq C e^{-\alpha\frac{t}{2}}$ by \eqref{eq_expa_eta_infty}.
			
			\item $\abs{\fD^2 \dot{\gamma}_0} \leq \abs{\fD^4 \gamma_0 } = o(1)$ by \eqref{eq_expa_gamma_0_infty}.
			
			\item $\abs{\fD^2 \dot{\gamma}_{1, k}} \leq C$, $\abs{\fD^2 \dot{\gamma}_{2, k}} \leq C$ by Proposition \ref{prop_bound_D2_dot_gamma_1_2}. 
			
			\item $\abs{\fD^2 \dot{\eta}_{2, k}} \leq \abs{\fD^4 {\eta}_{2, k}} \leq C e^{-\alpha \frac{t}{2} }$ by \eqref{eq_expa_eta_infty}.
			
		\end{enumerate}
		Therefore, $\abs{\fD^2 \Ric(\ob(t))}_{g(t)} \leq C$. By uniform equivalence between $g(t)$ and $g^\bullet(t)$ (see Lemma \ref{lem_metric_equiv}), the proof is complete.		
	\end{proof}

	\begin{rmk}
		By the \KR flow \eqref{eq_krf},  we have on $B \times Y$
		\begin{equation}
			\Ric(\ob(t)) = - \ocan - \ddbar \br{\varphi(t) + \dot{\varphi}(t)}.
		\end{equation}
		Combined with \eqref{eq_scal_evol}, the estimate on $\abs{\partial_t \Ric}$ in \Cref{thm_D^2_ric} implies in particular
		\begin{equation}
			\sup_{B \times Y \times \infint} \abs{\ddbar R(\ob(t))}_{\gb(t)} \leq C,
		\end{equation}
		which is weaker than the estimate on full real Hessian of the scalar curvature in \Cref{thm_nabla^2_scalar} proved below.
	\end{rmk}

	\subsection{\texorpdfstring{Proof of \Cref{thm_nabla^2_scalar}}{Proof of Theorem 1.2}}

	We calculate the difference  $\nabla^{\bullet, 2} \Ric - \D^2 \Ric$ by building upon the proof of \Cref{prop_D_nabla_b_Ric}. Using the difference tensor $A$ described in \eqref{eq_A_z_formula}, \eqref{eq_nabla_b-D_formula}, we have
	\begin{equation}
		\begin{aligned}
			\nabla^{\bullet, 2} \Ric - \D^2 \Ric
			& = \br{\nabla^\bullet - \D} \br{\nabla^\bullet \Ric} + \D \br{ \br{\nabla^\bullet - \D} \Ric} \\
			& = A \circledast \br{\nabla^\bullet \Ric} + \D \br{ A \circledast \Ric},
		\end{aligned}
	\end{equation}
	where
	\begin{equation}
		\br{A \circledast \Ric}_{jk\l} = {A}_{jk}^a \Ric_{a\l} + {A}_{j\l}^a \Ric_{ka}.
	\end{equation}
	By \eqref{eq_A_z_formula},
	\begin{equation}\label{eq_D_A}
		\begin{aligned}
			\br{\D A}_{ijk}^a
			 = & - \frac{1}{2} \br{\D_i g^{\bullet ab}} \br{\D_j g^\bullet_{kb} + \D_k g^\bullet_{jb} -\D_b g^\bullet_{jk} } \\
			&  - \frac{1}{2} {g^{\bullet ab}} \br{\br{\D^2 g^\bullet}_{ijkb} + \br{\D^2 g^\bullet}_{ikjb} -\br{\D^2 g^\bullet}_{ibjk} }.
		\end{aligned}
	\end{equation}
	Let us define the covariant 4-tensors $T, T'$ by
	\begin{equation}\label{eq_T_def}
		T_{ijk\l} = e^{-t}\Ric_{a\l} g^{\bullet ab} \sqbr{\br{\D^2 g_F}_{ijkb} + \br{\D^2 g_F}_{ikjb} - \br{\D^2 g_F}_{ibjk}}, \quad T'_{ijk\l} = T_{ij\l k}.
	\end{equation}
	Then combining all calculations above, we have 
	\begin{equation}\label{eq_nabla2-D2_Ric}
		\begin{aligned}
			{\nabla^{\bullet, 2} \Ric - \D^2 \Ric}
			= & A \circledast \br{\nabla^\bullet \Ric} + 
			\br{\D \br{\gb}^{-1}} \circledast \br{\D \gb} \circledast \Ric \\
			& + \br{\gb}^{-1} \circledast \br{\D^2 \br{\gb - e^{-t} g_F}} \circledast \Ric  - \frac{1}{2}(T-T') \\
			& + A \circledast \br{\D \Ric}.
		\end{aligned}
	\end{equation}
	Note the sign change between $T$ and $T'$ as $\Ric$ is a real $(1,1)$-form.

	\begin{proposition}\label{prop_nabla2_Ric_T}
		With tensors $T, T'$ defined in \eqref{eq_T_def}, we have
		\begin{equation}\label{eq_bound_nabla2_Ric_T}
			\abs{\nabla^{\bullet, 2} \Ric(\ob) + \frac{1}{2} \br{T-T'}} \leq C.
		\end{equation}
	\end{proposition}
	\begin{proof}
		Recall we have the following estimates.
		\begin{enumerate}
			\item $\abs{\br{\gb}^{-1}} \leq C$ by Lemma \ref{lem_metric_equiv}.
			
			\item $\abs{\D \gb} \leq C$, $\abs{\D \br{\gb}^{-1}} \leq C$ by Proposition \ref{prop_bound_D_gb}.
			
			\item $\abs{\Ric} \leq C$ by \eqref{eq_bound_ric}, which is \cite[Theorem 1.3]{HLT}.
			
			\item $\abs{\D \Ric} \leq C$ by \Cref{thm_D^1_ric}.
			
			\item $\abs{\nabla^{\bullet} \Ric} \leq C$ by \eqref{eq_bound_nabla_ric} or Theorem \ref{thm_nabla^1_ric}.
			
			\item $\abs{\D^2 \Ric} \leq C$ by \Cref{thm_D^2_ric}.
			
			\item $\abs{A} \leq C$ by \eqref{eq_bound_A_g}.
			
			\item $\abs{\D^2 \gcan} \leq C$ by Lemma \ref{lem_D_gcan}.
			
			\item $\abs{\D^2 \gamma_0^J} =o(1)$ by Lemma \ref{lem_D_gamma_0_J}.
			
			\item $\abs{\D^2 \gamma_{1, k}^J} \leq C$ by Lemma \ref{lem_D_gamma_1_J}.
			
			\item $\abs{\D^2 \eta_{1, k}^J} \leq Ce^{-\alpha\thalf}$ by Lemma \ref{lem_D_eta_1_J}.
			
		\end{enumerate}
		
		Therefore, writing $\gb - e^{-t} g_F = (1-e^{-t}) \gcan + \gamma_0^J + \gamma_{1, k}^J + \eta_{1, k}^J$ by \eqref{eq_g_bullet_j=1}, and plugging into \eqref{eq_nabla2-D2_Ric} the estimates above, we immediately conclude \eqref{eq_bound_nabla2_Ric_T}.
	\end{proof}

	The components of the base-fiber decomposition of $\Ric(\ob(t))$, $\gb(t)^{-1}$, satisfy the following estimates: 
	\begin{lemma}\label{lem_Ric_base_fiber}
		We have
		\begin{equation}\label{eq_Ric_bb}
			\abs{\Ric_{\bb \bb}}_{g_X} \leq C,
		\end{equation}
		\begin{equation}\label{eq_Ric_bf}
			\abs{\Ric_{\bb \ff}}_{g_X} \leq C e^{-3 \thalf},
		\end{equation}
		\begin{equation}\label{eq_Ric_ff}
			\abs{\Ric_{\ff \ff}}_{g_X} \leq C e^{-2t}.
		\end{equation}
	\end{lemma}
	\begin{proof}
		\eqref{eq_Ric_bb} is trivial as $\abs{\Ric} \leq C$ by \eqref{eq_bound_ric}. For the rest, recall that by the asymptotic expansion of $\ob$ in \eqref{eq_expa_j=2}, we can write
		\begin{equation}
			\Ric
			= - \ocan - \gamma_0 - \gamma_{1, k} - \eta_{1, k} - \dot{\gamma}_0 - \dot{\gamma}_{1, k} -\dot{ \gamma}_{2, k} - \dot{\eta}_{2, k}.
		\end{equation}
		Hence
		\begin{equation}
			\Ric_{\bb \ff} = \br{-\gamma_{1, k} - \eta_{1, k} - \dot{\gamma}_{1, k} -\dot{ \gamma}_{2, k} - \dot{\eta}_{2, k}}_{\bb \ff},
		\end{equation}
		and in $g(t)$-norm,
		\begin{equation}
			\abs{\Ric_{\bb \ff}} \leq \abs{\gamma_{1, k}} + \abs{\eta_{1, k}} + \abs{\dot{\gamma}_{1, k}} +\abs{\dot{ \gamma}_{2, k}} + \abs{\dot{\eta}_{2, k}} \leq C e^{-t},
		\end{equation}
		by Proposition \ref{prop_bound_D_gamma_1}, Proposition \ref{prop_bound_D2_dot_gamma_1_2}, and \eqref{eq_expa_eta_infty}. Thus \eqref{eq_Ric_bf} holds. Similarly, $\abs{\Ric_{\ff \ff}} \leq C e^{-t}$, and hence \eqref{eq_Ric_ff}.
	\end{proof}
	
	\begin{proposition}\label{prop_gb_base_fiber}
		We have
		\begin{equation}
			\abs{{\gb}^{\bb \bb}}_{g_X} \leq C,
		\end{equation}
		\begin{equation}
			\abs{{\gb}^{\bb \ff}}_{g_X} \leq C,
		\end{equation}
		\begin{equation}
			\abs{{\gb}^{\ff \ff}}_{g_X} \leq Ce^t.
		\end{equation}
	\end{proposition}
	\begin{proof}
		Recall from \eqref{eq_g_bullet_j=1} that
		\begin{equation}
			g^\bullet =  (1-e^{-t}) \gcan + e^{-t} g_F + \gamma_0^J + \gamma_{1, k}^J + \eta_{1, k}^J.
		\end{equation}
		We claim that
		\begin{equation}\label{eq_gb_bb}
			\gb_{\bb \bb} = \gcan + o_{g_X}(1),
		\end{equation}
		\begin{equation}\label{eq_gb_bf}
			\gb_{\bb \ff} = e^{-t} \br{g_F}_{\bb\ff} + O_{g_X}(e^{-3 \thalf}) = O_{g_X}(e^{-t}),
		\end{equation}
		\begin{equation}\label{eq_gb_ff}
			\gb_{\ff \ff} = e^{-t} \sqbr{\br{g_F}_{\ff\ff} + O_{g_X}(e^{-t})}.
		\end{equation}
		
		To see these, first recall that $\abs{J} = O(1)$, so $\abs{\gamma_0^J} \leq C\abs{\gamma_0} = o(1)$ by \eqref{eq_expa_gamma_0_infty}, $\abs{\gamma_{1, k}^J} \leq C\abs{\gamma_{1, k}} = O(e^{-t})$ by Proposition \ref{prop_bound_D_gamma_1}, and $\abs{\eta_{1, k}^J} \leq C\abs{\eta_{1, k}} = O(e^{-(2+\alpha) \thalf})$ by \eqref{eq_expa_eta_infty}. From these estimates we immediately deduce \eqref{eq_gb_bb}. To see \eqref{eq_gb_bf} and \eqref{eq_gb_ff}, observe that
		\begin{equation}
			\gb_{\bb\ff} = \br{e^{-t} g_F + \gamma_{1, k}^J + \eta_{1, k}^J}_{\bb\ff} = e^{-t} \br{g_F}_{\bb\ff} + O_{g_X}(e^{-3 \thalf}),
		\end{equation}
		and similarly,
		\begin{equation}
			\gb_{\ff\ff} = \br{e^{-t} g_F + \gamma_{1, k}^J + \eta_{1, k}^J}_{\ff\ff} = e^{-t} \br{g_F}_{\ff\ff} + O_{g_X}(e^{-2t}).
		\end{equation}
		Since $\gcan + \br{g_F}_{\ff \ff}$ defines a metric on $B \times Y$, we can compute $\br{g^\bullet}^{-1}$, in a product coordinate for all $t$ sufficiently large, using Schur complements of the block matrix
		\begin{equation}
			\gb = 
			\begin{bmatrix}
				\gb_{\bb \bb} & \gb_{\bb \ff} \\
				\gb_{\ff \bb} & \gb_{\ff \ff}
			\end{bmatrix}.
		\end{equation}
		In this process, \eqref{eq_gb_bb}, \eqref{eq_gb_bf}, \eqref{eq_gb_ff}  imply that as $t \to +\infty$, 
		\begin{equation}
			{\gb}^{\bb \bb} = \gcan^{-1} + o_{g_X}(1), 
		\end{equation}
		\begin{equation}
			{\gb}^{\bb \ff} = O_{g_X}(1), 
		\end{equation}
		\begin{equation}
			{\gb}^{\ff \ff} = e^t \br{\br{g_F}_{\ff\ff}^{-1} + O_{g_X}(e^{-t})}.
		\end{equation}
		We can cover $B \times Y$ by finitely many product coordinate neighborhoods up to shrinking $B$, and hence the proof is complete.
	\end{proof}
	
	\begin{proof}[Proof of \Cref{thm_nabla^2_scalar}]
		The scalar curvature of $\ob$ is given by $R = \tr{\gb}{\br{\Ric^J}} = \Ric \circledast J \circledast \br{\gb}^{-1}$. We apply this tensor contraction $\br{\cdot} \circledast J \circledast \br{\gb}^{-1}$ to \eqref{eq_bound_nabla2_Ric_T} to derive
		\begin{equation}\label{eq_nabla2_R_diff}
			\abs{\nabla^{\bullet, 2} R + \frac{1}{2} \br{T-T'} \circledast J \circledast \br{\gb}^{-1}} \leq C,
		\end{equation}
		where we use $\abs{J} \leq C$, $\abs{\br{\gb}^{-1}} \leq C$, $\nabla^\bullet J = 0$. By definition of $T$ in \eqref{eq_T_def}, we can write the tensor contraction in \eqref{eq_nabla2_R_diff} explicitly by
		\begin{equation}\label{eq_T_J_gb}
			\begin{aligned}
				\br{T \circledast J \circledast \br{\gb}^{-1}}_{ij} 
				& =  T_{ijkc} J_{\l}^c g^{\bullet k\l} \\
				& = e^{-t} S^{bk} \sqbr{\br{\D^2 g_F}_{ijkb} + \br{\D^2 g_F}_{ikjb} - \br{\D^2 g_F}_{ibjk}},
			\end{aligned}
		\end{equation}
		where $S$ is a contravariant 2-tensor defined by
		\begin{equation}\label{eq_S_bk}
			S^{bk} := \Ric_{ac} J_{\l}^c g^{\bullet ab}   g^{\bullet k\l},
		\end{equation}
		which is exactly the double sharp of the symmetric Ricci tensor $\Rc = \Ric^J$.

		We claim that
		\begin{equation}\label{eq_S_g_X}
			\abs{S}_{g_X} \leq C.
		\end{equation}
		To see this, decompose the tensor contraction
		\begin{equation}
			S = \Ric \circledast J \circledast \br{\gb}^{-1} \circledast \br{\gb}^{-1}
		\end{equation}
		using $\Ric = \Ric_{\bb \bb} + \Ric_{\bb \ff} + \Ric_{ \ff\bb} + \Ric_{\ff \ff}$ and analogous ones for $J$ and $\br{\gb}^{-1}$, and apply our estimates for these base-fiber specific components derived in Lemma \ref{lem_Ric_base_fiber} and Proposition \ref{prop_gb_base_fiber}. Consider the following cases, where we label the indices as in \eqref{eq_S_bk}:
		\begin{enumerate}
			\item $a, c\in \bb$. Then we can assume $\l \in \bb$ as $J_{\ff}^{\bb} = 0$. In this case we bound
			\begin{equation}
				\abs{\Ric_{\bb \bb}}_{g_X} \abs{J_{\bb}^{\bb}}_{g_X} \abs{g^{\bullet \bb *}}_{g_X} \abs{g^{\bullet *\bb}}_{g_X} \leq C,
			\end{equation}
			where $*$ denotes arbitrary $\bb$ or $\ff$.
			
			\item $a \in \ff$, $c \in \bb$. As above we assume $\l \in \bb$. Then we bound
			\begin{equation}
				\abs{\Ric_{\ff \bb}}_{g_X} \abs{J_{\bb}^{\bb}}_{g_X} \abs{g^{\bullet \ff *}}_{g_X} \abs{g^{\bullet *\bb}}_{g_X} \leq Ce^{-\thalf}.
			\end{equation}
			
			\item $a \in \bb$, $c \in \ff$. Then
			\begin{equation}
				\abs{\Ric_{\bb \ff}}_{g_X} \abs{J_{*}^{\ff}}_{g_X} \abs{g^{\bullet \bb *}}_{g_X} \abs{g^{\bullet **}}_{g_X} \leq C e^{-\thalf}.
			\end{equation}
			
			\item $a,c \in \ff$. Then 
			\begin{equation}
				\abs{\Ric_{\ff \ff}}_{g_X} \abs{J_{*}^{\ff}}_{g_X} \abs{g^{\bullet \ff *}}_{g_X} \abs{g^{\bullet **}}_{g_X} \leq C.
			\end{equation}
		\end{enumerate}
		Therefore, \eqref{eq_S_g_X} holds.
		
		We then deduce that the component $T \circledast J \circledast \br{\gb}^{-1}$ in \eqref{eq_nabla2_R_diff}, defined in \eqref{eq_T_J_gb} as a covariant 2-tensor, satisfies
		\begin{equation}\label{eq_T_J_gb_bound}
			\begin{aligned}
				\abs{T \circledast J \circledast \br{\gb}^{-1}} 
				& \leq C e^t \abs{T \circledast J \circledast \br{\gb}^{-1}}_{g_X} \\
				& \leq C \abs{S}_{g_X} \abs{\D^2 g_F}_{g_X} \\
				& \leq C,
			\end{aligned}
		\end{equation}
		by \eqref{eq_S_g_X} and $t$-independence of $\D^2 g_F$.
		
		The other component $T' \circledast J \circledast \br{\gb}^{-1}$ in \eqref{eq_nabla2_R_diff} can be bounded analogously, by estimating $S'$ instead of $S$ given by
		\begin{equation}
			\br{S'}^{cb} := \Ric_{ak} J_{\l}^c g^{\bullet ab} g^{\bullet k\l}.
		\end{equation}
		Again $\abs{S'}_{g_X} \leq C$ by Lemma \ref{lem_Ric_base_fiber} and Proposition \ref{prop_gb_base_fiber}, so that 
		\begin{equation}\label{eq_T'_J_gb_bound}
			\abs{T' \circledast J \circledast \br{\gb}^{-1}} \leq C \abs{S'}_{g_X} \abs{\D^2 g_F}_{g_X} \leq C.
		\end{equation}
		
		With \eqref{eq_T_J_gb_bound}, \eqref{eq_T'_J_gb_bound}, and uniform equivalence between $g(t)$ and $\gb(t)$ by Lemma \ref{lem_metric_equiv}, we can conclude from \eqref{eq_nabla2_R_diff} that
		\begin{equation}
			\abs{\nabla^{\ob(t), 2} R(\ob(t))}_{\gb(t)} \leq C.
		\end{equation}
		This completes the proof.
	\end{proof}

\subsection{\texorpdfstring{$\nabla^{\bullet, 2} \Ric$}{Hessian of Ricci} and Proof of \texorpdfstring{\Cref{thm_Delta_ric}}{Theorem 1.3}} We need the following facts on the base-fiber components of the curvature tensor of $\D$.

\begin{lemma}\label{lem_Rm_D}
	The curvature $\Rm^\D$ of $\D$, defined by 
	\begin{equation}
		\Rm^\D \br{X,Y,Z} := \D_X \D_Y Z - \D_Y \D_X Z - \D_{[X,Y]}Z,
	\end{equation}
	satisfies the following:
	\begin{equation}\label{eq_Rm_D_0}
		\br{\Rm^\D}_{\bb \bb \ff}^*, \br{\Rm^\D}_{\ff\ff\bb}^*, \br{\Rm^\D}_{\bb \bb \bb}^*, \br{\Rm^\D}_{\bb \ff \bb}^*, \br{\Rm^\D}_{\ff \bb \bb}^* \equiv 0, 
	\end{equation}
	\begin{equation}\label{eq_Rm_D_f}
		\br{\Rm^\D}_{\ff\ff\ff}^* = \br{\Rm^\D}_{\ff\ff\ff}^{\ff} = \Rm(g_{Y, z}), \quad \text{ on } \{z\} \times Y,
	\end{equation}
	\begin{equation}\label{eq_Rm_D_bf}
		\br{\Rm^\D}_{\bb \ff \ff}^* = \br{\Rm^\D}_{\bb \ff \ff}^{\ff}, \quad \br{\Rm^\D}_{ \ff \bb\ff}^* = \br{\Rm^\D}_{\ff \bb\ff}^{\ff}.
	\end{equation}
\end{lemma}
\begin{proof}
	By Definition \ref{def_D}, the connection $\D$ at $x = (z, y)$ concides with the Levi-Civita connection $\nabla^z$ for the product metric $g_z(0) = g_{\C^m} + g_{Y, z}$, where $g_{\C^m}$ is flat on $B$. We can thus easily conclude \eqref{eq_Rm_D_0}, \eqref{eq_Rm_D_f}, and \eqref{eq_Rm_D_bf}.
\end{proof}

\begin{proof}[Proof of \Cref{thm_Delta_ric}]

The evolution equation for the Ricci curvature along the normalized Ricci flow \eqref{eq_krf} is (see \cite[Corollary 7.3]{Ha})
\begin{equation}\label{eq_Ric_evol}
	\partial_t \Rc_{ij} = \frac{1}{2} \Delta \Rc_{ij} - g^{\bullet pq} \Rc_{rq} \Rm_{pij}^r - g^{\bullet pq} \Rc_{ip} \Rc_{jq},
\end{equation}
where $\Rc = \Ric^J$ denotes the symmetric Ricci tensor, and $\Rm$ the Riemann curvature $(1,3)$-tensor of $\gb$. 

Let $\T$ denote the covariant 2-tensor
\begin{equation}
	\T_{ij} := g^{\bullet pq} \Rc_{rq} \Rm_{pij}^r  = g^{\bullet pq} \Ric_{rc} J_q^c \Rm_{pij}^r ,
\end{equation}
the second term on the RHS of \eqref{eq_Ric_evol}. 
We claim that on $B \times Y \times \infint$,
\begin{equation}\label{eq_T_bound}
	\abs{\T} \leq C.
\end{equation}
To see this, we can write $\Rm$ in terms of $\Rm^\D$ and the difference tensor $A$ (see \eqref{eq_A_z_formula}, \eqref{eq_nabla_b-D_formula}). A standard calculation yields
\begin{equation}\label{eq_Rm_D}
	\Rm_{ijk}^\l = \br{\Rm^\D}_{ijk}^\l - \D_i A_{jk}^\l + \D_j A_{ik}^\l + A_{im}^\l A_{jk}^m - A_{jm}^\l A_{ik}^m,
\end{equation}
and accordingly we decompose $\T = \T_1 + \T_2 + \T_3$, where
\begin{equation}\label{eq_T_1}
	\T_{1,ij} = g^{\bullet pq} \Ric_{rc} J_q^c\br{\Rm^\D}_{pij}^r ,
\end{equation}
\begin{equation}
	\T_{2,ij} = g^{\bullet pq} \Ric_{rc} J_q^c \br{-\D_p A_{ij}^r + \D_i A_{pj}^r},
\end{equation}
\begin{equation}
	\T_{3,ij} = g^{\bullet pq} \Ric_{rc} J_q^c \br{A_{pm}^r A_{ij}^m -A_{im}^r A_{pj}^m}.
\end{equation}
We use base-fiber decomposition to bound these tensor contractions, with the help of \Cref{lem_Ric_base_fiber} and \Cref{prop_gb_base_fiber}, as in the proof of \Cref{thm_nabla^2_scalar}.

First, we establish
\begin{equation}\label{eq_T_1_bound}
	\abs{\T_1} \leq C, 
\end{equation}
by considering the following cases, with the same label of indices as in \eqref{eq_T_1}:
\begin{enumerate}
	\item $i, j \in \ff$. In this case we may assume $r \in \ff$ by \Cref{lem_Rm_D}. Since $J_{\ff}^{\bb} \equiv 0$, it suffices to bound
	\begin{equation}
		\abs{g^{\bullet **}}_{g_X} \abs{\Ric_{\ff \ff}}_{g_X} \abs{J_{*}^{\ff}}_{g_X} \abs{\br{\Rm^\D}_{*\ff\ff}^{\ff}}_{g_X} \leq C e^{-t},
	\end{equation}
	and
	\begin{equation}
		\abs{g^{\bullet *\bb}}_{g_X} \abs{\Ric_{\ff \bb}}_{g_X} \abs{J_{\bb}^{\bb}}_{g_X} \abs{\br{\Rm^\D}_{*\ff\ff}^{\ff}}_{g_X} \leq C e^{-3\thalf},
	\end{equation}
	using $t$-independence of $\Rm^\D$, so that
	\begin{equation}
		\abs{\T_{1,\ff\ff}} = e^t \abs{\T_{1,\ff\ff}}_{g_X} \leq C.
	\end{equation}

	\item $i \in \ff$, $j \in \bb$. Then, by \Cref{lem_Rm_D}, $\T_{1,\ff\bb} \equiv 0$.
	
	\item $i \in \bb$, $j \in \ff$. In this case we may assume $p, r \in \ff$ by \Cref{lem_Rm_D}. It suffices to bound
	\begin{equation}
		\abs{g^{\bullet \ff*}}_{g_X} \abs{\Ric_{\ff *}}_{g_X} \abs{J_{*}^{*}}_{g_X} \abs{\br{\Rm^\D}_{\ff\bb\ff}^{\ff}}_{g_X} \leq C e^{-\thalf},
	\end{equation}
	so that
	\begin{equation}
		\abs{\T_{1,\bb\ff}} = e^{\thalf} \abs{\T_{1,\bb\ff}}_{g_X} \leq C.
	\end{equation}

	\item $i, j \in \bb$. Then by \Cref{lem_Rm_D}, $\T_{1,\bb\bb} \equiv 0$.
\end{enumerate}
In conclusion, we have \eqref{eq_T_1_bound}.

Next, we prove
\begin{equation}\label{eq_T_2_bound}
	\abs{\T_2} \leq C.
\end{equation}
By \eqref{eq_D_A},
\begin{equation}
	\begin{aligned}
		\T_{2} 
		& = \sqbr{\br{\gb}^{-1} \circledast \Ric \circledast J \circledast \D \br{\gb}^{-1} \circledast \D \gb} \\
		& \quad + \sqbr{\br{\gb}^{-1} \circledast \Ric \circledast J \circledast \br{\gb}^{-1} \circledast \br{\D^2 \br{\gb - e^{-t}g_F}}}\\
		& \quad + e^{-t} \br{S \circledast \D^2 g_F},
	\end{aligned}
\end{equation}
where the contravariant 2-tensor $S$ is the one defined in \eqref{eq_S_bk}. We can now combine \Cref{lem_metric_equiv}, \Cref{lem_D_gcan,lem_D_gamma_0_J,lem_D_gamma_1_J,lem_D_eta_1_J}, \Cref{prop_bound_D_gb}, \eqref{eq_bound_ric}, \eqref{eq_g_bullet_j=1}, and \eqref{eq_S_g_X}, to conclude \eqref{eq_T_2_bound}. 

Finally, since $\abs{A} \leq C$ by \eqref{eq_bound_A_g}, we immediately have
\begin{equation}\label{eq_T_3_bound}
	\abs{\T_3} \leq C.
\end{equation}
Therefore, we see \eqref{eq_T_bound} by combining \eqref{eq_T_1_bound}, \eqref{eq_T_2_bound}, \eqref{eq_T_3_bound}.

Now it follows from \eqref{eq_Ric_evol} that
\begin{equation}
	\abs{\Delta \Rc} \leq C \br{\abs{\partial_t \Rc} + \abs{\T} + \abs{\Rc}^2} \leq C,
\end{equation}
using \Cref{thm_D^2_ric}, \eqref{eq_bound_ric}, and \eqref{eq_T_bound}. We can freely transition between the symmetric tensor $\Rc$ and the $(1,1)$-form $\Ric$ via the $t$-independent complex structure $J$, with $\abs{J} \leq C$. The proof is thus complete.
\end{proof}

\begin{rmk}
		According to Proposition \ref{prop_nabla2_Ric_T}, $T$ and $T'$ are the obstructions to the desired estimate $\abs{\nabla^{\bullet, 2} \Ric(\ob)}_{\gb} \leq C$. Although we have $\br{\D g_F}_{\ff \ff \ff} = 0$ from \eqref{eq_D_gF_fff}, and hence $\br{\D^2 g_F}_{* \ff \ff \ff} = 0$, it could happen that $\br{\D^2 g_F}_{\ff \bb \ff \ff}, \br{\D^2 g_F}_{\ff \ff \bb \ff}, \br{\D^2 g_F}_{\ff \ff \ff\bb }$ does not vanish. In that case we can only bound 
		\begin{equation}
			\abs{T_{\ff \ff \ff \bb}} \leq C e^{\frac{t}{2}}, \quad \abs{T_{***\ff}} \leq C,
		\end{equation}
		using Lemma \ref{lem_Ric_base_fiber} and Proposition \ref{prop_gb_base_fiber}, and hence $\abs{\nabla^{\bullet, 2} \Ric(\ob)}_{\gb} \leq Ce^{\thalf}$. One may wish to use $d$-closedness of $\omega_F$ to simplify $T$ and $T'$, but we still have to consider $\D^\ell J$, which describes the variation of the complex structure $J$ and may not vanish. As we shall see in \Cref{section_special}, when the Iitaka fibration is isotrivial or the generic fibers are tori, we do have uniform bounds on the covariant derivatives of $\Ric(\ob)$ of every order. In both cases, we have locally on $B \times Y$
		\begin{equation}\label{eq_J_parallel}
			\nabla^{g_{Y,z'}} J_{z} \equiv 0 \quad \text{ on } Y, \quad  \forall z, z' \in B.
		\end{equation}
		The read may check that \eqref{eq_J_parallel} is equivalent to the identical vanishing of $\br{\D^2 J}_{\ff\bb\ff}^{\ff}$ on $B \times Y$, which may be a quantity of interest in solving Conjecture \ref{conj_main}.
	\end{rmk}

\newpage
\section{Third-Order Ricci Estimates}\label{section_order_3}

In this section we prove \Cref{thm_D^3_Ric}. We work in the same local framework with the same simplification of notations as above.

Define functions $\cS$, $\E$ on $B \times Y$ by
\begin{equation}\label{eq_S_def}
	\frac{m}{n+1} \ocan^{m-1} \wedge \omega_F^{n+1} = \cS \ocan^m \wedge \omega_F^n,
\end{equation}
\begin{equation}\label{eq_error_full}
	\E := \sum_{p=1}^{N_{1, k}} A_{1, p, k} G_{1, p, k} + e^{-2t} \br{\cS - \underline{\cS}}.
\end{equation}
In \cite[\S 5.2]{HLT}, it was proved that
\begin{equation}
	\E = o(e^{-2t}),
\end{equation}
which suggests that $\davg{\cS}$, if not vanishing, would dominate the asymptotic behavior of certain objects related to $\gamma_{1, k}$. We exploit this idea and derive the following:
\begin{proposition}\label{prop_D3_gamma_1_Theta(1)}
	Suppose that $\davg{\cS}$ does not vanish identically. Then there exists $x \in B \times Y$ such that at $x$ and as $t \to +\infty$,
	\begin{equation}\label{eq_D^3_gamma_1_k}
		e^{-\frac{t}{2}} \abs{\D_{\ff \ff \ff }^3 (\gamma_{1, k} + \dot{\gamma}_{1, k})_{\ff \ff}}_{g(t)}(x, t) = \Theta(1).
	\end{equation} 
\end{proposition}

On the other hand, we can show that
\begin{proposition}\label{prop_D3_Ric_gamma_dot}
	We have
	\begin{equation}\label{eq_D3_Ric_gamma_dot}
		\abs{\D^3 \br{\Ric + \gamma_{1, k} + \dgamma_{1, k}}} = o\br{e^{\thalf}},
	\end{equation}
	and both $\gamma_{1, k}$ and $\dgamma_{1,k}$ satisfy
	\begin{equation}
		\abs{\D^3 \br{\cdot}} = O \br{e^{\thalf}}.
	\end{equation}
\end{proposition}
Combined, we conclude immediately \Cref{thm_D^3_Ric}: $\norm{\D^3 \Ric}$ would blow up at rate exactly ${e^{\thalf}}$ when $\cS$ defined by \eqref{eq_S_def} is not fiberwise constant. Therefore, we should in general not expect $t$-independent uniform bounds on the derivatives of $\Ric(\ob)$ of order $\geq 3$.

To prove \Cref{prop_D3_gamma_1_Theta(1),prop_D3_Ric_gamma_dot}, we shall apply Theorem \ref{thm_asymp_expa} with any fixed even integer $k \geq 10$ and $j \leq 5$. Note that all of the estimates in our previous sections, which only require $k \geq 4$, remain to hold on $B \times Y$ up to shrinking $B$. First,  we locate the dominant term in $\D^3_{\ff\ff\ff}\br{\gamma_{1, k} + \dgamma_{1, k}}_{\ff\ff}$. 

\begin{lemma}\label{lem_D^3_gamma_dot_gamma}
	We have
	\begin{equation}
		\abs{\D^3_{\ff \ff \ff} \br{\gamma_{1, k} + \dot{\gamma}_{1, k}}_{\ff \ff} - 
		\sum_{p=1}^{N_{1, k}} \br{A_{1, p, k} + \dot{A}_{1, p, k}} \D^3_{\ff \ff \ff} i\partial_{\ff} \overline{\partial}_{\ff} \br{\Delta^{\omega_F|_{\{\cdot\} \times Y}}}^{-1} G_{1, p, k}} 
		\leq C {e^{\br{1-\frac{\alpha}{3}}\frac{t}{2}} }.
	\end{equation}
\end{lemma}
\begin{proof}
	By \eqref{eq_gamma_i_k_ff}, we have
	\begin{equation}\label{eq_D^3_fff_gamma}
		\begin{aligned}
			& \D^3_{\ff \ff \ff} \br{\gamma_{1, k} + \dot{\gamma}_{1, k}}_{\ff \ff} 
			- \sum_{p=1}^{N_{1, k}} \br{A_{1, p, k} + \dot{A}_{1, p, k}} \D^3_{\ff \ff \ff}  \ddbarf \br{\Delta^{\omega_F|_{\{\cdot\} \times Y}}}^{-1} G_{1, p, k} \\
			& = \sum_{p=1}^{N_{1, k}} \sum_{(\iota, r) \neq (0, 0)} (1-r) e^{-rt} \br{\D^3_{\ff \ff \ff} \ddbarf \Phi_{\iota, r} \br{G_{1, p, k}}} \circledast \D^\iota A_{1, p, k} \\
			& \quad + \sum_{p=1}^{N_{1, k}} \sum_{ (\iota, r) \neq (0, 0)} e^{-rt} \br{\D^3_{\ff \ff \ff}  \ddbarf \Phi_{\iota, r} \br{G_{1, p, k}}} \circledast \D^\iota \dot{A}_{1, p, k}.
		\end{aligned}
	\end{equation}
	We estimate each summand in the RHS of \eqref{eq_D^3_fff_gamma}, where $0 \leq \iota \leq 2k$, $\ceil{\frac{\iota}{2}} \leq r \leq k$, and $(\iota, r) \neq (0, 0)$. Note first that
	\begin{equation}
		\abs{\D^3_{\ff \ff \ff}  \partial_{\ff} \overline{\partial}_{\ff} \Phi_{\iota, r} \br{G_{1, p, k}}} \leq C e^{5 \frac{t}{2}}.
	\end{equation}
	By Lemma \ref{lem_interp_fD_A_1},
	\begin{equation}\label{eq_bound_-rt_D_A_1}
		e^{-rt} \abs{\D^\iota A_{1, p, k}} \leq 
		\begin{cases}
			C e^{-2t}, & \iota = r = 0, \\
			C e^{-\br{5+\frac{\alpha}{5}} \thalf}, & 0 \leq \iota \leq 2, r = 1, \\
			C e^{-3t}, & r \geq 2.
		\end{cases}		
	\end{equation}
	By Lemma \ref{lem_interp_fD_dot_A_1_2},
	\begin{equation}\label{eq_bound_-rt_D_dot_A_1}
		e^{-rt} \abs{\D^\iota \dot{A}_{1, p, k}} \leq 
		\begin{cases}
			C e^{-2t}, & \iota = r = 0,\\
			C e^{-\br{2+\frac{\alpha}{3}}t}, & r \geq 1.
		\end{cases}
	\end{equation}
	We can then plug these estimates into \eqref{eq_D^3_fff_gamma}  to complete the proof.
\end{proof}

Thanks to Lemma \ref{lem_D^3_gamma_dot_gamma}, our next task is to estimate
\begin{equation}
	\abs{\sum_{p=1}^{N_{1, k}} \br{A_{1, p, k} + \dot{A}_{1, p, k}} \D^3_{\ff \ff \ff} i\partial_{\ff} \overline{\partial}_{\ff} \br{\Delta^{\omega_F|_{\{\cdot\} \times Y}}}^{-1} G_{1, p, k}}.
\end{equation} 
We approach this by applying the operator $\D^3_{\ff \ff \ff} i\partial_{\ff} \overline{\partial}_{\ff} \br{\Delta^{\omega_F|_{\{\cdot\} \times Y}}}^{-1}$ to 
\begin{equation}\label{eq_error_dot}
	\sum_{p=1}^{N_{1, k}} \br{A_{1, p, k} + \dot{A}_{1,p,k}} G_{1, p, k} = \E + \dot{\E} + e^{-2t} \br{\cS - \underline{\cS}},
\end{equation}
which follows from \eqref{eq_error_full}, and estimating the RHS of the equality thus derived from \eqref{eq_error_dot}.	Note that $\E$ and $\dot{\E}$ indeed have fiberwise average zero by definition \eqref{eq_error_full}. We aim to show the following.

\begin{proposition}\label{prop_error_o_-2t}
	$\E$ and $\dot{\E}$ satisfy
	\begin{equation}\label{eq_o_-2t_fffff}
		\abs{\D^3_{\ff \ff \ff} \ddbarf \br{\Delta^{\omega_F|_{\{\cdot\} \times Y}}}^{-1}\br{\cdot}}_{g_X} = o\br{e^{-2t}}.
	\end{equation}
\end{proposition}

The following facts about the operator $\D^3_{\ff \ff \ff} \ddbarf \br{\Delta^{\omega_F|_{\{\cdot\} \times Y}}}^{-1}$ will be useful.

\begin{lemma}[Uniform Fiberwise Schauder]\label{lem_unif_schauder}
	There exists $C > 0$ such that for any smooth function $F$ on $B \times Y$ with fiberwise average zero, 
	\begin{equation}\label{eq_schauder_fiber}
		\norm{\D^3_{\ff \ff \ff} \ddbarf \br{\Delta^{\omega_F|_{\{\cdot\} \times Y}}}^{-1} F }_{C^0(B \times Y, g_X)} 
		\leq C {\sum_{\l = 0}^4 \norm{\D^\l_{\ff \dots \ff} F}_{C^0(B \times Y, g_X)}  }.
	\end{equation}
	Therefore, \eqref{eq_o_-2t_fffff} follows from
	\begin{equation}\label{eq_o_-2t_l_leq_4}
		\abs{\D^\l_{\ff \dots \ff}(\cdot)}_{g_X} = o(e^{-2t}), \quad \forall 0 \leq \l \leq 4.
	\end{equation}
\end{lemma}
\begin{proof}
	Observe that the operator $\D_{\ff}$, when restricted to a fiber $\{z\} \times Y$, coincides with the Levi-Civita connection of $g_{Y, z}$ (cf. \eqref{eq_D_gF_nabla}). Apply Schauder estimates to $F|_{\{z\} \times Y}$ on each fiber to get
	\begin{equation}
		\begin{aligned}
			\norm{\D^3_{\ff \ff \ff} \ddbarf \br{\Delta^{\omega_F|_{\{\cdot\} \times Y}}}^{-1} F }_{C^0(\{z\} \times Y, g_{Y, z})}
			& \leq C \norm{\br{\Delta^{\omega_F|_{\{\cdot\} \times Y}}}^{-1} F }_{C^{5, \alpha}(\{z\} \times Y, g_{Y, z})} \\
			& \leq C \norm{F}_{C^{3, \alpha}(\{z\} \times Y, g_{Y, z})} \\
			& \leq C \norm{F}_{C^{4}(\{z\} \times Y, g_{Y, z})}, 
		\end{aligned}
	\end{equation}
	for some $C > 0$ depending on $J_z$ and $g_{Y, z}$. Since $J$ is smooth on $B \times Y$ and $\{g_{Y, z} \mid z \in B\}$ is a smooth family of metrics, we can piece together these fiberwise inequalities to derive \eqref{eq_schauder_fiber} up to shrinking $B$. 
\end{proof}

\begin{lemma}[Maximum Principle]\label{lem_max_principle}
	Let $(Y, g, \omega)$ be a compact Kähler manifold. Let $\nabla$ denote the Levi-Civita connection of $g$. If $F \in C^\infty(Y)$ has average zero and $\nabla^3(\ddbar \Delta^{-1} F) \equiv 0$, then $F \equiv 0$.
\end{lemma}
\begin{proof}
	Let $\eta := (\ddbar \Delta^{-1} F) \circ (\id \otimes J)$. Since $\nabla J = 0$, we have $\nabla^3 \eta = 0$. Tracing the last two entries of $\nabla^3 \eta$ by $g$, we see $0 = \nabla^3 (\tr{g}{\eta}) = \nabla^3 F$. Tracing again, we see $0 = \nabla \br{\tr{g}{\nabla^2 F}} = \nabla \br{\Delta F}$. Then $\Delta F$ is constant, so $\Delta F = 0$. $F$ has average zero, so $F \equiv 0$. 
\end{proof}

To prove Proposition \ref{prop_error_o_-2t}, we make the following preparations.

\begin{lemma}\label{lem_interp_fD_dot_A_new}
	For all $1 \leq p \leq N_{1, k}$,
	\begin{equation}\label{eq_interp_fD_dot_A_1}
		\abs{\fD^\l \dot{A}_{1, p, k}} \leq 
		\begin{cases}
			C e^{- \frac{t}{2} \br{4 -\frac{2\l}{8+\alpha}}}, 	& 0 \leq \l \leq 8, \\
			C e^{- \frac{t}{2} \br{4 -\frac{\l(\l-6)}{\l+\alpha}}}, & 8 \leq \l \leq 2k+10.
		\end{cases}
	\end{equation}
	
	For all $1 \leq p \leq N_{2, k}$,
	\begin{equation}\label{eq_interp_fD_dot_A_2}
		\abs{\fD^\l \dot{A}_{2, p, k}} \leq 
		\begin{cases}
			C e^{- \frac{t}{2} \br{4 + \beta -\frac{(2+\beta)\l}{8+\alpha}}}, 	& 0 \leq \l \leq 8, \\
			C e^{- \frac{t}{2} \br{4  -\frac{\l(\l-6) - \alpha \beta}{\l+\alpha}}}, & 8 \leq \l \leq 2k+10,
		\end{cases}
	\end{equation}
	for some $\beta = \beta(\alpha) \in (0, \alpha)$ if $\alpha > \frac{4}{5}$.
	
	For all $1 \leq p \leq N_{3, k}$,
	\begin{equation}\label{eq_interp_fD_A_3}
		\abs{\fD^\l {A}_{3, p, k}} \leq 
		\begin{cases}
			C e^{- \frac{t}{2} \br{6+\alpha -\frac{(4+\alpha)\l}{10+\alpha}}}, 	& 0 \leq \l \leq 10, \\
			C e^{- \frac{t}{2} \br{6+\alpha -\frac{\l(\l-6+\alpha)}{\l+\alpha}}}, & 10 \leq \l \leq 2k+12,
		\end{cases}
	\end{equation}
	\begin{equation}\label{eq_interp_fD_dot_A_3}
		\abs{\fD^\l \dot{A}_{3, p, k}} \leq 
		\begin{cases}
			C e^{- \frac{t}{2} \br{5 + \beta -\frac{(3+\beta)\l}{8+\alpha}}}, 	& 0 \leq \l \leq 8, \\
			C e^{- \frac{t}{2} \br{5-\frac{\l(\l-5)-\alpha \beta}{\l+\alpha}}}, & 8 \leq \l \leq 2k+10,
		\end{cases}
	\end{equation}
	for some $\beta = \beta(\alpha) \in (0, \alpha)$.
\end{lemma}
\begin{proof}
	We apply Theorem \ref{thm_asymp_expa} with $j = 5$. By \Cref{prop_bound_dot_A_i_leq2}, 
	\begin{equation}
		\abs{\dot{A}_{1,p,k}} \leq C e^{-2t}.
	\end{equation}
	By \eqref{eq_expa_A_i_infty}, 
	\begin{equation}
		\abs{\dot{A}_{2, p, k}} \leq \abs{\fD^2 A_{2,p,k}} \leq C e^{-(4+\alpha)\br{1-\frac{2}{12+\alpha}} \frac{t}{2}} \leq C e^{-(4+\beta) \frac{t}{2}},
	\end{equation}
	for some $\beta = \beta(\alpha) \in (0, \alpha)$ if $\alpha > \frac{4}{5}$, and
	\begin{equation}
		\abs{{A}_{3, p, k}} \leq  C e^{-(6+\alpha) \frac{t}{2}},
	\end{equation}
	\begin{equation}
		\abs{\dot{A}_{3, p, k}} \leq  C e^{-(6+\alpha)\br{1-\frac{2}{12+\alpha}} \frac{t}{2}} \leq C e^{-(5+\beta) \frac{t}{2}},
	\end{equation}
	for some $\beta = \beta(\alpha) \in (0, \alpha)$. By \eqref{eq_expa_A_i_holder}, $\sqbr{\fD^8 \dot{A}_{i,p,k}}_\alpha \leq \sqbr{\fD^{10} {A}_{i,p,k}}_\alpha \leq C e^{-t}$, and for all $8 \leq l \leq 2k+10$, $\sqbr{\fD^l \dot{A}_{i,p,k}}_\alpha \leq \sqbr{\fD^{l+2} {A}_{i,p,k}}_\alpha \leq C e^{(l-10)\thalf}$. Parabolic interpolation then completes the proof.
\end{proof}

From now on we choose $\alpha > \frac{4}{5}$ for parabolic \Ho (semi)norms such that \eqref{eq_interp_fD_dot_A_2} holds. Then we have the following estimates on $\gamma_{i,k}$.

\begin{lemma}\label{lem_bound_D^l_gamma_i_ff}
	$\br{\gamma_{1, k}}_{\ff \ff}$ and $\br{\dot{\gamma}_{1, k}}_{\ff \ff}$ satisfy
	\begin{equation}\label{eq_O_-2t}
		\abs{\D^\l_{\ff \dots \ff} \br{\cdot}}_{g_X} = O \br{e^{-2t}}, \quad 
		\forall \l \geq 0.
	\end{equation}
	$\br{\gamma_{2, k}}_{\ff \ff}$ and $\br{\dot{\gamma}_{2, k}}_{\ff \ff}$ satisfy 
	\begin{equation}
		\abs{\D^\l_{\ff \dots \ff} \br{\cdot}}_{g_X} = o \br{e^{-2t}},  \quad \forall \l \geq 0.
	\end{equation}
	$\br{\gamma_{3, k}}_{\ff \ff}$ and $\br{\dot{\gamma}_{3, k}}_{\ff \ff}$,  satisfy
	\begin{equation}
		\abs{\D^\l_{\ff \dots \ff} \br{\cdot}}_{g_X} = o \br{e^{-5 \frac{t}{2}}}, \quad \forall \l \geq 0.
	\end{equation}
\end{lemma}
\begin{proof}
	By \eqref{eq_gamma_i_k_ff},
	\begin{equation}\label{eq_D^l_gamma_1_ff}
		\D^\l_{\ff \dots \ff} \br{\gamma_{i, k}}_{\ff \ff} = 
		\sum_{p=1}^{N_{i, k}} \sum_{\iota = 0}^{2k} \sum_{r = \ceil{\frac{\iota}{2}}}^k 
		e^{-rt} \D^\l_{\ff \dots \ff} i\partial_{\ff} \overline{\partial}_{\ff} \Phi_{\iota, r} \br{G_{i, p, k}} \circledast \D^\iota A_{i, p, k},
	\end{equation}
	so
	\begin{equation}
		\abs{\D^\l_{\ff \dots \ff} \br{\gamma_{i, k}}_{\ff \ff}}_{g_X} \leq
		C \sum_{p=1}^{N_{i, k}} \sum_{\iota = 0}^{2k} \sum_{r = \ceil{\frac{\iota}{2}}}^k 
		e^{-rt} \abs{\D^\iota A_{i, p, k}},
	\end{equation}
	\begin{equation}
		\abs{\D^\l_{\ff \dots \ff} \br{\dgamma_{i, k}}_{\ff \ff}}_{g_X} \leq
		C \sum_{p=1}^{N_{i, k}} \sum_{\iota = 0}^{2k} \sum_{r = \ceil{\frac{\iota}{2}}}^k 
		e^{-rt} \br{\abs{\D^\iota A_{i, p, k}} + \abs{\D^\iota \dot{A}_{i, p, k}}}.
	\end{equation}
	We can then use \eqref{eq_bound_-rt_D_A_1}, \eqref{eq_bound_-rt_D_dot_A_1} for $i = 1$, Lemma \ref{lem_interp_fD_A_2}, Lemma \ref{lem_interp_fD_dot_A_new} for $i = 2, 3$, to complete the proof.
\end{proof}

\begin{lemma}\label{lem_bound_D^l_gamma_i_full}
	${\gamma_{1, k}}$, ${\dot{\gamma}_{1, k}}$ satisfy
	\begin{equation}\label{eq_o_-t}
		\abs{\D^\l_{\ff \dots \ff} \br{\cdot}}_{g_X} = o \br{e^{-t}},  \quad \forall \l \geq 0.
	\end{equation}
	${\gamma_{2, k}}$, ${\dot{\gamma}_{2, k}}$ satisfy
	\begin{equation}\label{eq_o_-3t/2}
		\abs{\D^\l_{\ff \dots \ff} \br{\cdot}}_{g_X} = o \br{e^{-3\thalf}},  \quad \forall \l \geq 0.
	\end{equation}
	${\gamma_{3, k}}$, ${\dot{\gamma}_{3, k}}$ satisfy
	\begin{equation}\label{eq_o_-2t}
		\abs{\D^\l_{\ff \dots \ff} \br{\cdot}}_{g_X} = o \br{e^{-2t}},  \quad \forall \l \geq 0.
	\end{equation}
\end{lemma}
\begin{proof}
	Apply operator $\D^\l_{\ff \dots \ff}$ to $\gamma_{i,k}$ given explicitly by \eqref{eq_D^q_gamma_i_k}, to get
	\begin{equation}
		\D^\l_{\ff \dots \ff} \gamma_{i, k} = \sum_{p=1}^{N_{i, k}} \sum_{\iota = 0}^{2k} \sum_{r = \ceil{\frac{\iota}{2}}}^k \sum_{s = 0}^{1} \sum_{i_1 + i_2 = s+1}
		e^{-rt} \D^\l_{\ff \dots \ff}  \sqbr{\br{\D^{1-s} J} \circledast \D^{i_1} \Phi_{\iota, r} \br{G_{i, p, k}}} \circledast \D^{i_2 + \iota} {A}_{i, p, k},
	\end{equation}
	so
	\begin{equation}
		\abs{\D^\l_{\ff \dots \ff} \gamma_{i, k}}_{g_X}
		 \leq C \sum_{p=1}^{N_{i, k}} \sum_{\iota = 0}^{2k} \sum_{r = \ceil{\frac{\iota}{2}}}^k \sum_{i_2 = 0}^{2} e^{-rt} \abs{\D^{i_2 + \iota}A_{i, p, k}},
	\end{equation}
	and similarly
	\begin{equation}
		\abs{\D^\l_{\ff \dots \ff} \dgamma_{i, k}}_{g_X}
		 \leq C \sum_{p=1}^{N_{i, k}} \sum_{\iota = 0}^{2k} \sum_{r = \ceil{\frac{\iota}{2}}}^k \sum_{i_2 = 0}^{2} e^{-rt} \br{\abs{\D^{i_2 + \iota}A_{i, p, k}}+\abs{\D^{i_2 + \iota}\dot{A}_{i, p, k}}  }.
	\end{equation}
	We can then use \Cref{lem_interp_fD_A_1,lem_interp_fD_A_2,lem_interp_fD_dot_A_new} to complete the proof.
\end{proof}

\begin{lemma}\label{lem_bound_D^l_eta_3}
	$\br{\eta_{3, k}}_{\ff \ff}$, $\br{\dot{\eta}_{3, k}}_{\ff \ff}$ satisfy 
	\begin{equation}
		\abs{\D^\l_{\ff \dots \ff} \br{\cdot}}_{g_X} = o \br{e^{-3t}}, \quad \forall 0 \leq \l \leq 4.
	\end{equation}
	${\eta_{3, k}}$, ${\dot{\eta}_{3, k}}$ satisfy 
	\begin{equation}
		\abs{\D^\l_{\ff \dots \ff} \br{\cdot}}_{g_X} = o \br{e^{-2t}}, \quad \forall 0 \leq \l \leq 4.
	\end{equation}
\end{lemma}
\begin{proof}
	By \eqref{eq_expa_eta_infty},
	\begin{equation}
		\begin{aligned}
			\abs{\D^\l_{\ff \dots \ff} \br{\eta_{3, k}}_{\ff \ff}}_{g_X}
			& \leq C e^{-(\l+2) \frac{t}{2}} \cdot \abs{\D^\l_{\ff \dots \ff} \br{\eta_{3, k}}_{\ff \ff}} \\
			& \leq C e^{-(\l+2) \frac{t}{2}} \cdot \abs{\fD^\l \eta_{3, k}} \\
			& \leq C e^{-(8+\alpha) \frac{t}{2}},
		\end{aligned}
	\end{equation}
	and similarly,
	\begin{equation}
		\abs{\D^\l_{\ff \dots \ff} \br{\dot{\eta}_{3, k}}_{\ff \ff}}_{g_X}
		 \leq C e^{-(\l+2) \frac{t}{2}} \cdot \abs{\fD^{\l+2} {\eta}_{3, k}} \leq C e^{-(6+\alpha) \frac{t}{2}},
	\end{equation}
	\begin{equation}
		\abs{\D^\l_{\ff \dots \ff} {\eta_{3, k}}}_{g_X}
		 \leq C e^{-\l \frac{t}{2}} \cdot \abs{\fD^\l \eta_{3, k}}
		 \leq C e^{-(6+\alpha) \frac{t}{2}},
	\end{equation}
	\begin{equation}
		\abs{\D^\l_{\ff \dots \ff} {\dot{\eta}_{3, k}}}_{g_X}
		 \leq C e^{-\l \frac{t}{2}} \cdot \abs{\fD^{\l+2} {\eta}_{3, k}}
		 \leq C e^{-(4+\alpha) \frac{t}{2}}.
	\end{equation}
\end{proof}

We have the following extension of \Cref{prop_davg_phi}:
\begin{lemma}\label{lem_phi-dot_phi}
	$\varphi - \underline{\varphi}$, $\dot{\varphi} - \underline{\dot{\varphi}}$, and $\ddot{\varphi} - \underline{\ddot{\varphi}}$ all satisfy 
	\begin{equation}\label{eq_D^l_f_o_-t}
		\abs{\D^\l_{\ff \dots \ff} \br{\cdot}}_{g_X} = o\br{e^{-t}}, \quad \forall 0 \leq \l \leq 4.
	\end{equation}
\end{lemma}
\begin{proof}
	Recall from \Cref{thm_asymp_expa}, \Cref{lem_green} that we can write
	\begin{equation}
		\varphi - \underline{\varphi} = G_{1, k} + G_{2, k} + G_{3, k} + \psi_{3, k},
	\end{equation}
	where for $1 \leq i \leq 3$,
	\begin{equation}
		G_{i, k} = \sum_{p=1}^{N_{i, k}}  \sum_{\iota = 0}^{2k} \sum_{r = \ceil{\frac{\iota}{2}}}^k e^{-rt} {\Phi_{\iota, r} (G_{i, p, k})} \circledast \D^\iota A_{i, p, k}
	\end{equation}
	has fiberwise average zero, and $\ddbar \psi_{3, k} = \eta_{3, k}$. We then have for all $\l \geq 0$,
	\begin{equation}
		\abs{\D^\l_{\ff \dots \ff} G_{i, k}}_{g_X} \leq C \sum_{p=1}^{N_{i, k}}  \sum_{\iota = 0}^{2k} \sum_{r = \ceil{\frac{\iota}{2}}}^k e^{-rt}  \abs{\D^\iota A_{i, p, k}},
	\end{equation}	
	\begin{equation}
		\abs{\D^\l_{\ff \dots \ff} \dot{G}_{i, k}}_{g_X} \leq C \sum_{p=1}^{N_{i, k}}  \sum_{\iota = 0}^{2k} \sum_{r = \ceil{\frac{\iota}{2}}}^k e^{-rt}  \br{\abs{\D^\iota A_{i, p, k}} + \abs{\D^\iota \dot{A}_{i, p, k}}},
	\end{equation}	
	\begin{equation}
		\abs{\D^\l_{\ff \dots \ff} \ddot{G}_{i, k}}_{g_X} \leq C \sum_{p=1}^{N_{i, k}}  \sum_{\iota = 0}^{2k} \sum_{r = \ceil{\frac{\iota}{2}}}^k e^{-rt}  \br{\abs{\D^\iota A_{i, p, k}} + \abs{\D^\iota \dot{A}_{i, p, k}} + \abs{\D^\iota \ddot{A}_{i, p, k}}}.
	\end{equation}	
	By \eqref{eq_bound_-rt_D_A_1}, \eqref{eq_bound_-rt_D_dot_A_1}, \Cref{lem_interp_fD_A_1}, \Cref{lem_interp_fD_A_2},  \Cref{lem_interp_fD_dot_A_new}, we see that $G_{i, k}, \dot{G}_{i, k}, \ddot{G}_{i, k}$ satisfy \eqref{eq_D^l_f_o_-t} for all $1 \leq i \leq 3$.

	To deal with $\psi_{3, k}$ (which has fiberwise average zero), apply Schauder estimates on each fiber $(\{z\} \times Y, g_{Y, z})$ to get for each $0 \leq \l \leq 4$,
	\begin{equation}
		\begin{aligned}
			& \norm{\D^\l_{\ff \dots \ff} \psi_{3, k}}_{C^0(\{z\} \times Y, g_{Y, z})} \\
			& \leq C \norm{\Delta^{\omega_F|_{\{z\} \times Y}} \psi_{3, k}}_{C^{2, \alpha}(\{z\} \times Y, g_{Y, z})} \\
			& \leq C \norm{\br{\eta_{3, k}}_{\ff \ff}}_{C^{2, \alpha}(\{z\} \times Y, g_{Y, z})} \\
			& \leq C \br{\sum_{p=0}^3 e^{-(p+2)\thalf}\norm{ \D^p_{\ff\dots\ff}\br{\eta_{3, k}}_{\ff\ff}}_{\infty, B \times Y \times [t-1, t], g(t)} }.
		\end{aligned}
	\end{equation}
	Since $\{g_{Y, z} \mid z \in B\}$ is a smooth family of metrics, we can piece together these fiberwise estimates to derive
	\begin{equation}
		\abs{\D^\l_{\ff\dots\ff} \psi_{3,k}}_{g_X} \leq C e^{-(8 + \alpha)\thalf}, \quad \forall 0 \leq \l \leq 4,
	\end{equation}
	thanks to \eqref{eq_expa_eta_infty}. We can replace $\psi_{3, k}$ by $\dot{\psi}_{3, k}$ in the argument above to derive
	\begin{equation}
		\abs{\D^\l_{\ff\dots\ff} \dot{\psi}_{3,k}}_{g_X} \leq C e^{-(6 + \alpha)\thalf}, \quad \forall 0 \leq \l \leq 4.
	\end{equation}
	The case of $\ddot{\psi}_{3, k}$ is slightly different and uses in addition \eqref{eq_expa_eta_holder}:
	\begin{equation}
		\begin{aligned}
			& \norm{\D^\l_{\ff \dots \ff} \ddot{\psi}_{3, k}}_{C^0(\{z\} \times Y, g_{Y, z})} \\
			& \leq C \sum_{p=0}^2 e^{-(p+2)\thalf}\norm{ \fD^{p+4} \eta_{3, k}}_{\infty, B \times Y \times [t-1, t], g(t)}  \\
			& \quad + C e^{-(4+\alpha)\thalf} \sqbr{\fD^{6} \eta_{3, k}}_{\alpha, \alpha/2, B \times Y \times [t-1, t], g(t)},
		\end{aligned}
	\end{equation}
	so
	\begin{equation}
		\abs{\D^\l_{\ff\dots\ff} \ddot{\psi}_{3,k}}_{g_X} \leq C e^{-(4 + \alpha)\thalf}, \quad \forall 0 \leq \l \leq 4.
	\end{equation}
	This completes the proof.
\end{proof}

We are now ready to prove Proposition \ref{prop_error_o_-2t}. Using the asymptotic expansion Theorem \ref{thm_asymp_expa}, we write the parabolic \MA equation \eqref{eq_krf_MA}  for the \KR flow as
\begin{equation}\label{eq_flow_MA}
	e^{\varphi + \dot{\varphi}} \omega_{\can}^m \wedge \omega_F^n
	= \frac{e^{nt}}{\binom{m+n}{n}} \sqbr{(1-e^{-t}) \omega_{\can} + e^{-t} \omega_F + \gamma_0 + \gamma_{1, k} + \gamma_{2, k} + \gamma_{3,k} + \eta_{3, k}}^{m+n}.
\end{equation}
Define 
\begin{equation}
	\beta:= (1-e^{-t}) \omega_{\can} + \gamma_0,
\end{equation}
as a form on the base, and expand \eqref{eq_flow_MA} as
\begin{align}
	& e^{\varphi + \dot{\varphi}} \omega_{\can}^m \wedge \omega_F^n \label{eq_error_left} \\
	& = \beta^m \sqbr{\omega_F + e^t \br{\gamma_{1, k} + \gamma_{2, k} + \gamma_{3, k} + \eta_{3, k}}}_{\ff \ff}^n \notag\\
	& \quad + \sum_{q=1}^m \frac{m! n!}{(m-q)! (n+q)!} 
	\beta^{m-q} 
	e^{-qt} \sqbr{\omega_F + e^t \br{\gamma_{1, k} + \gamma_{2, k} + \gamma_{3, k} + \eta_{3, k}}}^{n+q} \notag \\
	& = \beta^m \sqbr{\br{\omega_F}_{\ff \ff}^n + e^t n \br{\omega_F}_{\ff \ff}^{n-1} \br{\gamma_{1, k}}_{\ff \ff}} \label{eq_error_1} \\
	& \quad + \beta^m \sqbr{e^t n \br{\omega_F}_{\ff \ff}^{n-1} \br{\gamma_{2, k} + \gamma_{3, k} + \eta_{3, k}}_{\ff \ff}} \label{eq_error_2} \\
	& \quad + \beta^m \sum_{r=2}^n \binom{n}{r} \br{\omega_F}_{\ff \ff}^{n-r} \sqbr{e^{t}  \br{\gamma_{1, k} + \gamma_{2, k} + \gamma_{3, k} + \eta_{3, k}}_{\ff \ff}}^r \label{eq_error_3} \\
	& \quad + \frac{m}{n + 1} \beta^{m-1} e^{-t} \omega_F^{n+1} \label{eq_error_4} \\
	& \quad + \frac{m}{n + 1} \beta^{m-1} e^{-t} \sum_{r=1}^{n+1} \binom{n+1}{r} \omega_F^{n+1-r} \sqbr{e^{t}   \br{\gamma_{1, k} + \gamma_{2, k} + \gamma_{3, k} + \eta_{3, k}}}^r  \label{eq_error_5} \\
	& \quad + \sum_{q=2}^m \frac{m! n!}{(m-q)! (n+q)!} 
	\beta^{m-q} 
	e^{-qt} \sqbr{\omega_F + e^t \br{\gamma_{1, k} + \gamma_{2, k} + \gamma_{3, k} + \eta_{3, k}}}^{n+q}	\label{eq_error_6}.
\end{align}
Define an operator $\T$ on top-degree forms on $B \times Y$ by: 1. dividing by $e^t \omega_{\can}^m \wedge \omega_F^n$; 2. subtracting from the resulting function its fiberwise average to get a function on $B \times Y$. We will apply $\T$ to the expansion above and analyze line-by-line. As an overview, $\T\eqref{eq_error_1}$ contains $\sum_{p=1}^{N_{1, k}} A_{1, p, k} G_{1, p, k}$, $\T\eqref{eq_error_4}$ contains $e^{-2t} \br{\cS - \underline{\cS}}$, and the remaining terms constitute the error $\E$ defined in \eqref{eq_error_full}.

First, by $\de\dbar$-exactness of $\gamma_{1,k}$,
\begin{equation}
	\T\eqref{eq_error_1} 
	= \B \tr{\omega_F|_{\{\cdot\} \times Y}}{ \br{\gamma_{1, k}}_{\ff \ff}},
\end{equation}
where
\begin{equation}
	  \B := \frac{\beta^m}{\ocan^m}.
\end{equation}
Thus by \eqref{eq_gamma_i_k_ff}, \eqref{eq_error_1} contributes to $\E$ by $\E_1 := \E_{1, 1} + \E_{1,2}$, where
\begin{equation}
	\E_{1, 1} := \br{\B - 1} \sum_{p=1}^{N_{1, k}} A_{1, p, k} G_{1, p, k},
\end{equation}
\begin{equation}
	\E_{1, 2} := \B  \sum_{p=1}^{N_{1, k}} \sum_{(\iota, r) \neq (0, 0)} e^{-rt} \Delta_{\omega_F|_{\{\cdot\} \times Y}}\br{\Phi_{\iota, r}(G_{1, p, k})} \circledast \D^\iota A_{1, p, k}.
\end{equation}

\begin{lemma}\label{lem_E_1}
	$\E_1$ and $\dot{\E}_1$ satisfy \eqref{eq_o_-2t_fffff}.
\end{lemma}
\begin{proof}
	By \eqref{eq_expa_gamma_0_infty},
	\begin{equation}\label{eq_B}
		\B = 1 + o(1)_{\base}, \quad \dot{\B} = o(1)_{\base}.
	\end{equation}
	Note that since $A_{1,p,k}$ live on the base, the derivatives $\D^\l_{\ff \dots \ff}$, when applied to $\E_{1, i}$, will only land on the functions involving $G_{1, p, k}$ to form some $t$-independent tensor. Therefore, by \eqref{eq_bound_-rt_D_A_1}, \eqref{eq_bound_-rt_D_dot_A_1}, we see that $\E_{1, 1}$, $\dot{\E}_{1, 1}$, $\E_{1, 2}$, $\dot{\E}_{1, 2}$ all satisfy \eqref{eq_o_-2t_l_leq_4}, and hence \eqref{eq_o_-2t_fffff} by Lemma~\ref{lem_unif_schauder}. 
\end{proof}

Next, by  $\de\dbar$-exactness of $\gamma_{i, k}$ and $\eta_{i, k}$,
\begin{align}
	\T \eqref{eq_error_2} 
	& = \B \tr{\omega_F|_{\{\cdot\} \times Y}}{ \br{\gamma_{2, k} + \gamma_{3, k} + \eta_{3, k}}_{\ff \ff}} =: \E_2.
\end{align}
Thus \eqref{eq_error_2} contributes to $\E$ by $\E_2$.
\begin{lemma}\label{lem_E_2}
	$\E_2$ and $\dot{\E}_2$ satisfy \eqref{eq_o_-2t_fffff}.
\end{lemma}
\begin{proof}
	Write $\E_2 = \E_{2, 2} + \E_{2, 3} + \E_{2, 4}$, where for $i = 2, 3$,
	\begin{equation}
		\E_{2, i} := \B \tr{\omega_F|_{\{\cdot\} \times Y}}{ \br{\gamma_{i, k}}_{\ff \ff}} = \B \sum_{p=1}^{N_{i, k}} \sum_{\iota = 0}^{2k} \sum_{r = \ceil{\frac{\iota}{2}}}^k  
		e^{-rt} \Delta_{\omega_F|_{\{\cdot\} \times Y}}\br{\Phi_{\iota, r}(G_{i, p, k})} \circledast \D^\iota A_{i, p, k},
	\end{equation}
	and
	\begin{equation}
		\E_{2, 4} := \B \tr{\omega_F|_{\{\cdot\} \times Y}}{ \br{\eta_{3, k}}_{\ff \ff}}.
	\end{equation}
	We can then use \eqref{eq_B}, \Cref{lem_interp_fD_A_2,lem_interp_fD_dot_A_new}, to conclude that $\E_{2, i}$ and $\dot{\E}_{2,i}$ satisfy \eqref{eq_o_-2t_l_leq_4} for $i = 2, 3$, and hence \eqref{eq_o_-2t_fffff}.

	As for $\E_{2, 4}$, recall we can write $\eta_{3, k} = \ddbar \psi_{3, k}$ for some function $\psi_{3, k}$ having fiberwise average zero. Then
	\begin{equation}
		\begin{aligned}
			\abs{\D^3_{\ff \ff \ff} i\partial_{\ff} \overline{\partial}_{ \ff} \br{\Delta^{\omega_F|_{\{\cdot\} \times Y}}}^{-1} \E_{2, 4}}_{g_X}
			& \leq C e^{-5 \frac{t}{2}} \abs{\D^3 \eta_{3, k}} \\
			& \leq C {e^{-(8+\alpha) \frac{t}{2}}}, 
		\end{aligned}
	\end{equation} 
	by \eqref{eq_expa_eta_infty}. Similarly,
	\begin{align*}
		\abs{\D^3_{\ff \ff \ff} i\partial_{\ff} \overline{\partial}_{ \ff} \br{\Delta^{\omega_F|_{\{\cdot\} \times Y}}}^{-1} \dot{\E}_{2, 4}}_{g_X} \leq C {e^{-(6+\alpha) \frac{t}{2}}} .
	\end{align*}
	Thus $\E_{2, 4}$ and $\dot{\E}_{2, 4}$ satisfies \eqref{eq_o_-2t_fffff}. This completes the proof.
\end{proof}

Next, by  $\de\dbar$-exactness of $\gamma_{i, k}$ and $\eta_{i, k}$,
\begin{equation}
	\T \eqref{eq_error_3} = e^{-t} \B \sum_{r=2}^n \binom{n}{r} \frac{ \br{\omega_F}_{\ff \ff}^{n-r} \sqbr{e^{t}  \br{\gamma_{1, k} + \gamma_{2, k} + \gamma_{3, k} + \eta_{3, k}}_{\ff \ff}}^r}{\br{\omega_F}_{\ff \ff}^n} =: \E_3.
\end{equation}
Thus \eqref{eq_error_3} contributes to $\E$ by $\E_3$.
\begin{lemma}\label{lem_E_3}
	$\E_3$ and $\dot{\E}_3$ satisfy \eqref{eq_o_-2t_fffff}.
\end{lemma}
\begin{proof}
	By \Cref{lem_bound_D^l_gamma_i_ff,lem_bound_D^l_eta_3}, we see that $e^{t}  \br{\gamma_{1, k} + \gamma_{2, k} + \gamma_{3, k} + \eta_{3, k}}_{\ff \ff}$ and its $t$-derivative satisfy
	\begin{equation}
		\abs{\D^\l_{\ff \dots \ff} \br{\cdot}}_{g_X} = O(e^{-t}), \quad \forall 0 \leq \l \leq 4.
	\end{equation}
	We can use the product rule for $\D_{\ff}$  to see that for each $r \geq 2$, 
	\begin{equation}
		\frac{ \br{\omega_F}_{\ff \ff}^{n-r} \sqbr{e^{t}  \br{\gamma_{1, k} + \gamma_{2, k} + \gamma_{3, k} + \eta_{3, k}}_{\ff \ff}}^r}{\br{\omega_F}_{\ff \ff}^n}
	\end{equation}
	and its $t$-derivative satisfy
	\begin{equation}
		\abs{\D^\l_{\ff \dots \ff} \br{\cdot}}_{g_X} = O(e^{-2t}), \quad \forall 0 \leq \l \leq 4.
	\end{equation}	
	Combined with \eqref{eq_B},  we see that $\E_{3}$ and $\dot{\E}_{3}$ satisfy
	\begin{equation}
		\abs{\D^\l_{\ff \dots \ff} \br{\cdot}}_{g_X} = O(e^{-3t}), \quad \forall 0 \leq \l \leq 4,
	\end{equation}
	Thus $\E_{3}$ and $\dot{\E}_{3}$ satisfy \eqref{eq_o_-2t_l_leq_4}, and hence \eqref{eq_o_-2t_fffff}.
\end{proof}

Next, we have
\begin{equation}
	\T \eqref{eq_error_4} = e^{-2t} \br{\cS - \underline{\cS} + \tilde{\E}_4 - \underline{\tilde{\E}_4}},
\end{equation}
where
\begin{equation}
	\tilde{\E}_4 := \frac{m}{n+1} \frac{\beta_1 \wedge \omega_F^{n+1}}{\omega_{\can}^m \wedge \omega_F^n}, \quad \beta_1 := \beta^{m-1} - \omega_{\can}^{m-1}.
\end{equation}
Thus \eqref{eq_error_4} contributes to $\E$ by $\E_4 := e^{-2t} \br{\tilde{\E}_4 - \underline{\tilde{\E}_4}}$.

\begin{lemma}\label{lem_E_4}
	$\E_4$ and $\dot{\E}_4$ satisfy \eqref{eq_o_-2t_fffff}.
\end{lemma}
\begin{proof}
	Since $\beta_1$ is a form on the base, for any $\l \geq 0$,
	\begin{equation}
		\D^\l_{\ff \dots \ff} \br{\beta_1 \wedge \omega_F^{n+1}} = \beta_1 \wedge \D^\l_{\ff \dots \ff}  \omega_F^{n+1}.
	\end{equation}
	Also, $\abs{\beta_1}_{g_X} = o(1)$, $\abs{\dot{\beta}_1}_{g_X} = o(1)$ by \eqref{eq_expa_gamma_0_infty}. We can thus use the product rule for $\D_{\ff}$ to handle the quotient of forms in $\tilde{\E}_4$ and see that $\tilde{\E}_4$ and $\dot{\tilde{\E}}_4$ satisfy
	\begin{equation}
		\abs{\D^\l_{\ff \dots \ff} \br{\cdot}}_{g_X} = o(1), \quad \forall \l \geq 0.
	\end{equation}
	Hence ${\E}_4$ and $\dot{{\E}}_4$ satisfy
	\begin{equation}
		\abs{\D^\l_{\ff \dots \ff} \br{\cdot}}_{g_X} = o \br{e^{-2t}}, \quad \forall \l \geq 0,
	\end{equation}
	which implies \eqref{eq_o_-2t_fffff} by \Cref{lem_unif_schauder}.
\end{proof}

Next, we have
\begin{equation}
	\T \eqref{eq_error_5} = \sum_{r=1}^{n+1} e^{-2t} \br{\tilde{\E}_{5, r} - \underline{\tilde{\E}_{5, r}}} =: \E_5,  
\end{equation}
where
\begin{equation}
	\tilde{\E}_{5, r} := \frac{m}{n+1} \binom{n+1}{r} \frac{\beta^{m-1}  \omega_F^{n+1-r} \sqbr{e^t\br{\gamma_{1, k} + \gamma_{2, k} + \gamma_{3, k} + \eta_{3, k}}}^r}{\omega_{\can}^m \wedge \omega_F^n}.
\end{equation}
Thus \eqref{eq_error_5} contributes to $\E$ by $\E_5$.

\begin{lemma}\label{lem_E_5}
	$\E_5$ and $\dot{\E}_5$ satisfy \eqref{eq_o_-2t_fffff}.
\end{lemma}
\begin{proof}
	Note that $\beta$ is a form on the base, and $\abs{\beta}_{g_X} =O(1)$, $\abs{\dot{\beta}}_{g_X} =o(1)$ by \eqref{eq_expa_gamma_0_infty}. By \Cref{lem_bound_D^l_gamma_i_full,lem_bound_D^l_eta_3}, we see that ${\beta^{m-1}  \omega_F^{n+1-r} \sqbr{e^t\br{\gamma_{1, k} + \gamma_{2, k} + \gamma_{3, k} + \eta_{3, k}}}^r}$ and its $t$-derivative satisfy, for each $r \geq 1$,
	\begin{equation}\label{eq_o_1}
		\abs{\D^\l_{\ff \dots \ff} \br{\cdot}}_{g_X} = o(1), \quad \forall 0 \leq \l \leq 4.
	\end{equation}
	We can thus use the product rule for $\D_{\ff}$ to see that $\tilde{\E}_{5, r}$ and $\dot{\tilde{\E}}_{5, r}$ satisfy \eqref{eq_o_1}. Therefore, $\E_5$ and $\dot{{\E}}_5$ satisfy \eqref{eq_o_-2t_l_leq_4}, and hence \eqref{eq_o_-2t_fffff}.
\end{proof}

Next, we have
\begin{equation}
	\T\eqref{eq_error_6} = \sum_{q=2}^{m} e^{-(q+1)t} \br{\tilde{\E}_{6, q} - \underline{\tilde{\E}_{6, q}}} =: \E_6,  
\end{equation}
where
\begin{equation}
	\tilde{\E}_{6, q} := \frac{m! n!}{(m-q)! (n+q)!}  \frac{\beta^{m-q}  \sqbr{\omega_F + e^t\br{\gamma_{1, k} + \gamma_{2, k} + \gamma_{3, k} + \eta_{3, k}}}^{n+q}}{\omega_{\can}^m \wedge \omega_F^n},
\end{equation}
Thus \eqref{eq_error_6} contributes to $\E$ by $\E_6$.

\begin{lemma}\label{lem_E_6}
	$\E_6$ and $\dot{\E}_6$ satisfy \eqref{eq_o_-2t_fffff}.
\end{lemma}
\begin{proof}
	The proof is completely analogous to that of Lemma \ref{lem_E_5}. Since $q \geq 2$, we can in fact see that $\E_6$ and $\dot{\E}_6$ satisfy
	\begin{equation}
		\abs{\D^\l_{\ff \dots \ff} \br{\cdot}}_{g_X} = O \br{e^{-3t}}, \quad \forall 0 \leq \l \leq 4.
	\end{equation}
\end{proof}

We are left with the contribution $\E_7$ to $\E$ by \eqref{eq_error_left}:
\begin{equation}
	\E_7 := \T \br{e^{\varphi + \dot{\varphi}} \omega_{\can}^m \wedge \omega_F^n } = e^{-t} \br{e^{\varphi + \dot{\varphi}} - \underline{e^{\varphi + \dot{\varphi}}}}.
\end{equation}
\begin{lemma}\label{lem_E_7}
	$\E_7$ and $\dot{\E}_7$ satisfy \eqref{eq_o_-2t_fffff}.
\end{lemma}
\begin{proof}
	We first consider $\E_7$. Thanks to \Cref{lem_known_est} and \Cref{lem_phi-dot_phi}, we can use Taylor expansion of exponential to estimate
	\begin{equation}\label{eq_exp_phi_avg}
		\abs{e^{\varphi + \dot{\varphi}} - \underline{e^{\varphi + \dot{\varphi}}}} = o\br{e^{-t}}.
	\end{equation}
	For each $1 \leq \l \leq 4$, we can decompose $\D^\l_{\ff \dots \ff} \br{e^{\varphi + \dot{\varphi}}}$ by product rule and use Lemma \ref{lem_phi-dot_phi} to see that
	\begin{equation}
		\abs{\D^\l_{\ff \dots \ff} \br{e^{\varphi + \dot{\varphi}} - \underline{e^{\varphi + \dot{\varphi}}}}}_{g_X} = o\br{e^{-t}}, \quad \forall 1 \leq \l \leq 4.
	\end{equation}
	Therefore, $\E_7$ satisfies 
	\eqref{eq_o_-2t_l_leq_4}, and hence \eqref{eq_o_-2t_fffff}.
	
	It remains to consider $\dot{\E}_7$. Since
	\begin{equation}
		\partial_t \br{e^{\varphi + \dot{\varphi}} - \underline{e^{\varphi + \dot{\varphi}}}}
		= {e^{\varphi + \dot{\varphi}} \br{\dot{\varphi} + \ddot{\varphi}} - \underline{e^{\varphi + \dot{\varphi}} \br{\dot{\varphi} + \ddot{\varphi}}}},
	\end{equation}
	we can apply the idea in \eqref{eq_prod_avg_diff} to get
	\begin{align*}
		\abs{\partial_t \br{e^{\varphi + \dot{\varphi}} - \underline{e^{\varphi + \dot{\varphi}}}}} = o \br{e^{-t}},
	\end{align*}
	where we use \eqref{eq_exp_phi_avg}, \Cref{lem_known_est},  \Cref{lem_phi-dot_phi}. For each $1 \leq \l \leq 4$, decompose 
	\begin{equation}
		\D^\l_{\ff \dots \ff} \partial_t \br{e^{\varphi + \dot{\varphi}} - \underline{e^{\varphi + \dot{\varphi}}}} = \D^\l_{\ff \dots \ff} \br{e^{\varphi + \dot{\varphi}} \br{\dot{\varphi} + \ddot{\varphi}}}
	\end{equation}
	by product rule and use Lemma \ref{lem_phi-dot_phi} to see that
	\begin{equation}
		\abs{\D^\l_{\ff \dots \ff} \partial_t \br{e^{\varphi + \dot{\varphi}} - \underline{e^{\varphi + \dot{\varphi}}}} }_{g_X} = o\br{e^{-t}}, \quad \forall 1 \leq \l \leq 4.
	\end{equation}
	Therefore, $\dot{\E}_7$ satisfies 
	\eqref{eq_o_-2t_l_leq_4}, and hence \eqref{eq_o_-2t_fffff}.
\end{proof}

In summary, according to definition \eqref{eq_error_full}, we can write
\begin{equation}
	\E = -\br{\E_1 + \dots + \E_6} + \E_7.
\end{equation}
Combining \Cref{lem_E_1,lem_E_2,lem_E_3,lem_E_4,lem_E_5,lem_E_6,lem_E_7}, we immediately have the following:
\begin{lemma}\label{lem_E}
	$\E$ and $\dot{\E}$ satisfy \eqref{eq_o_-2t_fffff}.
\end{lemma}

We are now ready to prove Proposition \ref{prop_D3_gamma_1_Theta(1)}.

\begin{proof}[Proof of \Cref{prop_D3_gamma_1_Theta(1)}]
	Since $\cS$ is time-independent, \eqref{eq_error_full} implies
	\begin{equation}\label{eq_error_full_dot}
		\sum_{p=1}^{N_{1, k}} \br{A_{1, p, k} + \dot{A}_{1, p, k}} G_{1, p, k} = \E + \dot{\E} + e^{-2t} \br{\cS - \underline{\cS}}.
	\end{equation}
	Apply the operator $\D^3_{\ff \ff \ff} \partial_{\ff} \overline{\partial}_{ \ff} \br{\Delta^{\omega_F|_{\{\cdot\} \times Y}}}^{-1} $ to \eqref{eq_error_full_dot}. On one fiber $\{z\} \times Y$ where $\cS$ is not constant, we can pick some point $x$ such that 
	\begin{equation}\label{eq_non_zero_dom}
		\D^3_{\ff \ff \ff} \partial_{\ff} \overline{\partial}_{ \ff} \br{\Delta^{\omega_F|_{\{\cdot\} \times Y}}}^{-1} \br{\cS - \underline{\cS}}(x) \neq 0
	\end{equation}
	due to Lemma \ref{lem_max_principle}.
	Using Lemma \ref{lem_E} and time-independence of \eqref{eq_non_zero_dom}, we see that
	\begin{equation}
		e^{- \frac{t}{2}} \abs{\sum_{p=1}^{N_{1, k}} \br{A_{1, p, k} + \dot{A}_{1, p, k}} \D^3_{\ff \ff \ff} \ddbarf \br{\Delta^{\omega_F|_{\{\cdot\} \times Y}}}^{-1}  G_{1, p, k}}_{g(t)}(x, t) = \Theta(1).
	\end{equation}
	Combine this with Lemma \ref{lem_D^3_gamma_dot_gamma} to get
	\begin{equation}
		e^{- \frac{t}{2}} \abs{\D^3_{\ff \ff \ff} \br{\gamma_{1, k} + \dot{\gamma}_{1, k}}_{\ff \ff}}_{g(t)}(x, t) = \Theta(1).
	\end{equation}
	This completes the proof.
\end{proof}

To prove \Cref{prop_D3_Ric_gamma_dot}, we make the following preparations. 

\begin{lemma}\label{lem_D3_gamma_i_leq3}
	$\gamma_{1,k}$ and $\dgamma_{1,k}$ satisfy
	\begin{equation}
		\abs{\D^3 \br{\cdot}} = O\br{e^{\thalf}}.
	\end{equation}
	For both $i = 2, 3$, $\gamma_{i,k}$ and $\dgamma_{i,k}$ satisfy
	\begin{equation}
		\abs{\D^3 \br{\cdot}} = o\br{e^{\thalf}}.
	\end{equation}
\end{lemma}
\begin{proof}
	For all $1 \leq i \leq 3$, we have by \eqref{eq_D^q_gamma_i_k}
	\begin{equation}
		\begin{aligned}
			\abs{\D^3 \gamma_{i, k}} 
			& \leq \sum_{p=1}^{N_{i, k}} \sum_{\iota = 0}^{2k} \sum_{r = \ceil{\frac{\iota}{2}}}^k \sum_{s = 0}^{4} \sum_{i_1 + i_2 = s+1}
			e^{\thalf\br{-2r + 4-s + i_1}} \abs{\D^{i_2 + \iota} {A}_{i, p, k}} \\
			& \leq \sum_{p=1}^{N_{i, k}} \sum_{\iota = 0}^{2k} \sum_{r = \ceil{\frac{\iota}{2}}}^k \sum_{i_2 = 0}^{5} 
			e^{\thalf\br{-2r + 5-i_2}} \abs{\D^{i_2 + \iota} {A}_{i, p, k}},
		\end{aligned}
	\end{equation}
	and similarly
	\begin{equation}
		\abs{\D^3 \dgamma_{i, k}} 
		\leq \sum_{p=1}^{N_{i, k}} \sum_{\iota = 0}^{2k} \sum_{r = \ceil{\frac{\iota}{2}}}^k \sum_{i_2 = 0}^{5} 
		e^{\thalf\br{-2r + 5-i_2}} \br{\abs{\D^{i_2 + \iota} {A}_{i, p, k}} + \abs{\D^{i_2 + \iota} \dot{A}_{i, p, k}}}.
	\end{equation}
	We can then use \Cref{lem_interp_fD_A_1,lem_interp_fD_A_2,lem_interp_fD_dot_A_new} to complete the proof.
\end{proof}

\begin{proof}[Proof of \Cref{prop_D3_Ric_gamma_dot}]
	By the \KR flow \eqref{eq_krf} and Theorem \ref{thm_asymp_expa}, we can write
	\begin{equation}\label{eq_Ric_1_2}
		\Ric +\gamma_{1, k} +\dot{\gamma}_{1, k}= - \omega_{\can} - \gamma_0 -   \gamma_{2, k} - \eta_{2, k} - \dot{\gamma}_0  - \dot{\gamma}_{2, k} - \dot{\gamma}_{3, k} - \dot{\eta}_{3, k}.
	\end{equation}
	We estimate the $\D^3$-derivative of each term in the RHS of \eqref{eq_Ric_1_2}.
	\begin{enumerate}
		\item $\abs{\D^3 \omega_{\can}} \leq C$ since $\ocan$ lives on the base.
		
		\item $\abs{\D^3 \gamma_0} ,\abs{\D^3 \dot{\gamma}_0} = o(1)$ by \eqref{eq_expa_gamma_0_infty}.
		
		\item $\abs{\D^3 \gamma_{2, k}} ,\abs{\D^3 \dgamma_{2, k}} ,\abs{\D^3 \dgamma_{3, k}} =o \br{e^{\thalf}}$ by \Cref{lem_D3_gamma_i_leq3}. 
		
		\item $\abs{\D^3 \eta_{2, k}}, \abs{\D^3 \dot{\eta}_{3, k}} \leq C e^{-(1+\alpha) \frac{t}{2}}$ by \eqref{eq_expa_eta_infty}.
	\end{enumerate}
	Therefore, \eqref{eq_D3_Ric_gamma_dot} holds.	Together with \Cref{lem_D3_gamma_i_leq3}, the proof is complete.
\end{proof}

\begin{rmk}\label{rmk_moduli}
	We expect that $\cS - \underline{\cS}$ does not vanish in general. Below is a heuristic construction. From \cite[\S 5]{HT} we have the identity
	\begin{equation}\label{eq_S_A}
		\cS - \underline{\cS} = \br{\Delta^{\omega_F|_{\{\cdot\} \times Y}}}^{-1} \br{\gcan^{\mu \overline{\nu}} \br{\lan{A_\mu, A_{\overline{\nu}}} - \underline{\lan{A_\mu, A_{\overline{\nu}}}}} },
	\end{equation}
	where $A_\mu$ ($1 \leq \mu \leq m$) are the unique harmonic representatives of the Kodaira-Spencer classes $\kappa(\partial_\mu) \in H^{1}(X_z, T_{X_z})$ with respect to the Ricci-flat metric $\omega_F|_{X_z}$ on $X_z$, $\lan{\cdot, \cdot}$ is the $\omega_F|_{X_z}$-inner product on $T_{X_z} \otimes \Omega^{0,1}_{X_z}$, and $\partial_\mu$ are the standard coordinate vector fields on $B \subset \C^m$. If we set $m = 1$ and $n = 2$, we may start with a K3 surface $Y$ with a Ricci-flat metric $\omega_Y$ and a harmonic representative $A_1$ of some class $\kappa \in H^1(Y, T_Y)$ having non-constant $\omega_Y$-length (see e.g. \cite{CC,He1} for the asymptotically cylindrical gluing construction of K3 surfaces that admit such $A_1$). If $\omega_Y$ represents an ample line bundle $L$ on $Y$, which we fix as the polarization class, such that $\kappa$ respects $L$, then we may let the complex structure of $Y$ vary in the moduli space of K3 surfaces in the direction $\kappa$, to get a K3 fibration over a curve such that the total space $X$ is our desired compact \K manifold. One needs to check that the Iitaka fibration coincides with the construction above, and $X$ admits a \K metric $\omega_0$ whose restriction to the fiber $Y$ we started with is cohomologous to $\omega_Y$. If such construction can be made precise, it would follow immediately from \eqref{eq_S_A} that $\davg{\cS} \not\equiv 0$ on $Y$.
\end{rmk}

\section{Special Cases and Remarks}\label{section_special}
	In this section we consider, in the following special cases of the Iitaka fibration, higher-order curvature estimates over any $K \Subset X \setminus S$. Again we work locally on $B \times Y$.

	\subsection{Isotrivial Fibration}\label{section_isotrivial}
		Given the Iitaka fibration $f: X \setminus S \to B \setminus f(S)$ as in the Introduction, let us assume additionally that the fibers $X_z$ for $z \in B \setminus f(S)$ are pairwise biholomorphic (such $f$ is called \textit{isotrivial}). By the Fischer-Grauert theorem, $f$ is locally holomorphically trivial. As above, we can locally write  $f: B \times Y \to B$, and the complex structure $J$ on $B \times Y$ is now a product: $J = J_{\C^m} \oplus J_{Y}$. 
		
		In this case, Fong-Lee showed in \cite{FL} that up to shrinking $B$, for each $k \in \N$, there exists $C_k$ such that
		\begin{equation}\label{eq_nabla_gb}
			\sup_{B \times Y \times \infint} \abs{\nabla^{g(t), k} \gb(t)}_{g(t)} \leq C_k,
		\end{equation}
		where $g(t) = g_{\C^m} + e^{-t} g_{Y, 0}$ as in \eqref{eq_g_z(t)}. 
		Since $\gb(t)$ and $g(t)$ are uniformly equivalent by \Cref{lem_metric_equiv}, we have uniform bounds on the covariant derivatives of $\Ric(\gb)$ of every order:
		
		\begin{theorem}\label{thm_isotrivial}
			Under the assumptions above, for each $k \in \N$, there exists $C_k$ such that
			\begin{equation}
				\sup_{B \times Y \times \infint} \abs{\nabla^{\gb(t), k} \Ric \br{\gb(t)}}_{\gb(t)} \leq C_k.
			\end{equation}
		\end{theorem}
		\begin{proof}
			To simplify notation, let $\nabla$ denote $\nabla^{g(t)}$, and $\nabla^\bullet$ denote $\nabla^{\gb(t)}$.
			Let $A$ denote the difference (1,2)-tensor between $\nabla^\bullet$ and $\nabla$. Standard calculation yields
			\begin{equation}\label{eq_Rm_diff}
				\Rm(\gb) - \Rm(g) = \nabla A + A \circledast A,
			\end{equation}
			where $\Rm$ denotes the full Riemann curvature (1,3)-tensor, and
			\begin{equation}
				A = \br{\gb}^{-1} \circledast \nabla \gb.
			\end{equation}
			
			By \Cref{lem_metric_equiv} and \eqref{eq_nabla_gb}, $\abs{\nabla^k A } \leq C_k$ for each $k \in \N$. Tracing \eqref{eq_Rm_diff} and using definition of $A$ and Ricci-flatness of $g(t)$, we see that $\nabla^{\bullet, k} \Ric(\gb)$ is a linear combination of contractions of tensors $A, \nabla A, \dots, \nabla^{k+1} A$. Thus $\abs{\nabla^{\bullet, k} \Ric(\gb)} \leq C_k$, and the proof is complete.
		\end{proof}

	\subsection{Torus Fibers}
		In this section let us assume instead that for some $z_0 \in B \setminus f(S)$, the fiber $X_{z_0}$ is biholomorphic to the quotient of a complex torus by a holomorphic free action of a finite group. Then locally we write $f: B \times Y \to B$, where $z_0 = 0 \in B \subset \C^m$. Up to shrinking $B$, we have uniform bounds on the covariant derivatives of $\Rm(\gb)$ of every order:
		\begin{theorem}\label{thm_torus}
			Under the assumptions above, for each $k \in \N$, there exists $C_k$ such that
			\begin{equation}\label{eq_torus_Rm}
				\sup_{B \times Y \times \infint} \abs{\nabla^{\gb(t), k} \Rm \br{\gb(t)}}_{\gb(t)} \leq C_k.
			\end{equation}
		\end{theorem}
		\begin{proof}
			We build on \cite[Proof of Theorem 5.24]{To}. Suppose first that $X_{z_0}$ is biholomorphic to a complex torus. Up to shrinking $B$, we have a universal covering 
			\begin{equation}
				p: B \times \C^n \to B \times Y,
			\end{equation}
			which is $(J_E, J)$-holomorphic and satisfies $f \circ p(z, \cdot) = z$. Here $J_E$ denotes the Euclidean complex structure on $B \times \C^n$, and similarly let $g_E$ denote the Euclidean metric. Define $\lambda_t: B \times \C^n \to B \times \C^n$ by 
			\begin{equation}
				\lambda_t(z, y) = (z, e^{\thalf} y),
			\end{equation}
			which stretches the fibers. It was proved in \cite{To} that for each $K \Subset B \times \C^n$, there exists $C_K$  such that
			\begin{equation}
				C_K^{-1} g_E \leq \lambda_t^*p^*\gb(t) \leq C_K g_E, \quad \text{ on } K \times \infint,
			\end{equation}
			and for each $k \in \N$, there exists $C_{K, k}$ such that
			\begin{equation}
				\sup_{K \times [0, +\infty)} \abs{\nabla^{g_E, k} \lambda_t^*p^* \gb(t) }_{g_E} \leq C_{K, k}.
			\end{equation}
			It follows that
			\begin{equation}
				\sup_{K \times [0, +\infty)} \abs{\nabla^{\lambda_t^*p^*\gb(t), k} \Rm \br{\lambda_t^*p^*\gb(t)} }_{\lambda_t^*p^*\gb(t)} \leq C_{K, k},
			\end{equation}
			if we compare $\Rm\br{g_E}$ ($\equiv 0$) with $\Rm\br{\lambda_t^*p^*\gb(t)}$ by the idea in the proof of \Cref{thm_isotrivial}. 
			
			If now $K' \subset B \times Y$ is sufficiently small such that there exists $K \Subset B \times \C^n$ with $p: K \to K'$ a biholomorphism, then we have
			\begin{equation}
				\begin{aligned}
					\sup_{K'} \abs{\nabla^{\gb(t),k} \Rm(\gb(t))}_{\gb(t)}
					& = \sup_{K} \abs{\nabla^{p^*\gb(t),k} \Rm(p^*\gb(t))}_{p^*\gb(t)} \\
					& = \sup_{\lambda_{-t} \br{K}} \abs{\nabla^{\lambda_t^*p^*\gb(t),k} \Rm(\lambda_t^*p^*\gb(t))}_{\lambda_t^*p^*\gb(t)} \\
					& \leq C_{\tilde{K}, k}
				\end{aligned}
			\end{equation}
			as $\lambda_{-t}(K) \subset \tilde{K}$ for some $\tilde{K} \Subset B \times \C^n$, for all $t \geq 0$. Note that each point of $B \times Y$ admits such a neighborhood $K'$. Hence up to shrinking $B$, a covering argument shows \eqref{eq_torus_Rm}.
			
			When $X_{z_0}$ is only a finite quotient of a complex torus, we can reduce the problem to the complex torus case above using the argument in \cite{To}. In summary, construct a finite covering
			\begin{equation}
				p: B \times \tilde{Y} \to B \times Y,
			\end{equation}
			where $\tilde{Y}$ is a torus. The arguments above apply to the flow metric $p^* \gb(t)$ and fibration $f \circ p: B \times \tilde{Y} \to B$, and the estimates thus derived imply the desired ones on $B \times Y$ via the finite covering map $p$.
		\end{proof}
		
		\begin{rmk}
			We can also derive \Cref{thm_torus} from the local uniform bound on $\abs{\Rm(\gb(t))}_{\gb(t)}$  established in \cite[Theorem 5.24]{To}, using Shi's derivative estimates along Ricci flows (see \cite{Shi}).
		\end{rmk}

	\subsection{Trivial Iitaka Fibration}
		Suppose $(X, \omega_0)$ is a compact \K manifold with $K_X$ semiample and Kodaira dimension taking the extremal values 0 or $\dim X$. In both cases the Iitaka fibration is trivial (see \cite{To}): 
		\begin{enumerate}
			\item When $\kod(X) = 0$, we know that $c_1(X) = 0$, and $K_X^{ p} \cong \mathcal{O}_X$ is holomorphically trivial for some $p \geq 1$. Thus the base is a point and $X$ is the only Calabi-Yau fiber.
			
			\item When $\kod(X) = \dim X$, the generic fibers are connected and of dimension 0, so that $f: X \setminus S \to B \setminus f(S)$ is a biholomorphism. As $f^* \mathcal{O}_{\C \CP^N}(1) \cong K_X^{ p}$ for some $p \geq 1$, we know that $K_X$ is nef and big.
		\end{enumerate}
		
		These are extremal cases of the isotrivial fibration (where we assume intermediate Kodaira dimension) discussed in \Cref{section_isotrivial}, which motivates us to expect again uniform bounds on the covariant derivatives of $\Ric(\gb)$ of every order. We confirm this speculation now.
		
		Consider the case $\kod(X) = 0$. In \cite{Cao} Cao proved that the unnormalized \KR flow
		\begin{equation}
			\partial_t \omega(t) = -\Ric(\omega(t)), \quad \omega(0) = \omega_0,
		\end{equation}
		converges smoothly to the unique Ricci-flat \K metric $\omega_F$ in the class $[\omega_0]$. It follows from \cite{PS,TZ} that the convergence is exponentially fast in all $C^k$-norms. We can then compare $\Ric(g(t))$ with $\Ric(g_F)$ ($\equiv 0$) following the idea in the proof of \Cref{thm_isotrivial}, to see that there exists $\lambda > 0$ such that
		\begin{equation}
			\norm{\nabla^{g(t), k} \Ric(g(t))}_{C^0(X, g(t))} \leq C_k e^{-\lambda t}, 
		\end{equation}
		for each $k \in \N$. The normalized \KR flow writes
		\begin{equation}
			\ob(\tau) = e^{-\tau} \omega(e^\tau -1),
		\end{equation}
		from which we deduce that
		\begin{equation}
			\begin{aligned}
				\norm{\nabla^{\gb(\tau), k} \Ric(\gb(\tau))}_{C^0(X, \gb(\tau))} 
				& \leq C_k e^{-\lambda \br{e^\tau -1}} e^{(k+2)\frac{\tau}{2}} \\
				& \leq C_k' e^{-\lambda' e^\tau},
			\end{aligned}
		\end{equation}
		where $\lambda' > 0$. Therefore, in fact $\Ric(\gb)$ decays to zero fast in $C^k(X, \gb)$ for all $k$.
		
		Consider the other case $\kod(X) = \dim X$. By \cite{Ts,TiZ}, the normalized \KR flow $\ob(t)$ converges smoothly on $K \Subset X \setminus \mathrm{Null}(c_1(K_X))$ to some \KE metric $\omega_\infty$ satisfying $\Ric(\omega_\infty) = -\omega_\infty$ on $X \setminus \mathrm{Null}(c_1(K_X))$. Comparing $\Ric(\ob)$ with $\Ric(\omega_\infty) = -\omega_\infty$ following the idea in the proof of \Cref{thm_isotrivial}, we see that for any $K \Subset X \setminus \mathrm{Null}(c_1(K_X))$, $k \in \N$, 
		\begin{equation}
			\norm{\nabla^{\gb(t), k} \Ric(\gb(t))}_{C^0(K, \gb(t))} \leq C_{K, k}. 
		\end{equation}
		In fact, $\mathrm{Null}(c_1(K_X)) \subset S$. The discussions for these extremal cases are thus complete.

\end{document}